\documentclass[11pt]{amsart}

\usepackage{amsmath}
\usepackage{amssymb}
\usepackage{bbm}
\usepackage{pdfsync}
\usepackage{esint} 
\usepackage[breaklinks=true]{hyperref}
\usepackage[utf8]{inputenc}
\usepackage[utf8]{inputenc}
\usepackage[T1]{fontenc}

\newcommand{\R}{{\mathbb R}}       

\newcommand{\DD}{{\mathcal D}}
\newcommand{\FF}{{\mathcal F}}
\newcommand{\HH}{{\mathcal H}}
\newcommand{\LL}{{\mathcal L}}
\newcommand{\PP}{{\mathcal P}}

\newcommand{\Ch}{{\mathcal Ch}}

\newcommand{\TT}{{\mathcal T}}

\newcommand{\RR}{{\mathcal R}}

\newcommand{\EE}{{\mathcal E}}

\newcommand{\diam}{{\rm diam}}
\newcommand{\dist}{{\rm dist}}

\newcommand{\rf}[1]{{(\ref{#1})}}

\newcommand{\supp}{\operatorname{supp}}

\newcommand{\vphi}{{\varphi}}
\newcommand{\ve}{{\varepsilon}}
\newcommand{\vv}{{\vspace{2mm}}}

\newcommand{\wt}[1]{{\widetilde{#1}}}
\newcommand{\wh}[1]{{\widehat{#1}}}

\newcommand{\pom}{{\partial\Omega}}

\newcommand{\sss}{{\mathsf {Stop}}}
\newcommand{\ttt}{{\mathsf {Top}}}
\newcommand{\DB}{{\mathsf {DB}}}

\newcommand{\tree}{{\rm Tree}}

\newcommand{\pv}{\operatorname{pv}}
\newcommand{\rad}{\operatorname{rad}}

\newcommand{\HD}{{\mathsf{HD}}}
\newcommand{\hd}{{\mathsf{hd}}}

\newcommand{\MDW}{{\mathsf{MDW}}}

\newcommand{\LD}{{\mathsf{LD}}}

\newcommand{\Trc}{{\mathsf{Trc}}}
\newcommand{\sL}{{\mathsf{L}}}

\newcommand{\End}{{\mathsf{End}}}

\newcommand{\Ty}{{\mathsf{Ty}}}

\newcommand{\GDF}{{\mathsf{GDF}}}
\newcommand{\RE}{{\wt\RR_\ve}}

\newcommand{\capp}{{\operatorname{Cap}}}

\def\Xint#1{\mathchoice
{\XXint\displaystyle\textstyle{#1}}%
{\XXint\textstyle\scriptstyle{#1}}%
{\XXint\scriptstyle\scriptscriptstyle{#1}}%
{\XXint\scriptscriptstyle\scriptscriptstyle{#1}}%
\!\int}
\def\XXint#1#2#3{{\setbox0=\hbox{$#1{#2#3}{\int}$ }
\vcenter{\hbox{$#2#3$ }}\kern-.58\wd0}}

\def\avint{\;\Xint-}

\def\BMO{\mathop\mathrm{BMO}} 					
\def\H11{\textup{H}^1_1} 					

\usepackage{pgf,tikz} 

\definecolor{ffffff}{rgb}{1.0,1.0,1.0}
\definecolor{qqqqff}{rgb}{0.0,0.0,1.0}
\definecolor{ffqqqq}{rgb}{1.0,0.0,0.0}
\definecolor{zzzzqq}{rgb}{0.6,0.6,0.0}
\definecolor{marronet}{rgb}{0.6,0.2,0}
\definecolor{negre}{rgb}{0,0,0}
\definecolor{vermell}{rgb}{0.8,0.05,0.05}
\definecolor{blau}{rgb}{0.3,0.2,1.}
\definecolor{blauclar}{rgb}{0.,0.,1.}
\definecolor{grisfosc}{rgb}{0.25098039215686274,0.25098039215686274,0.25098039215686274}
\definecolor{verd}{rgb}{0.1,0.6,0.1}
\definecolor{taronja}{rgb}{0.9,0.6,0.05}
\definecolor{vermellclar}{rgb}{1.,0.,0.}
\definecolor{verdet}{rgb}{0,0.8,0.1}
\definecolor{blauverd}{rgb}{0,0.4,0.2}
\definecolor{grisclar}{rgb}{0.6274509803921569,0.6274509803921569,0.6274509803921569}

\newtheorem{theorem}{Theorem}[section]
\newtheorem{lemma}[theorem]{Lemma}
\newtheorem{mlemma}[theorem]{Main Lemma}
\newtheorem{keylemma}[theorem]{Key Lemma}
\newtheorem{coro}[theorem]{Corollary}
\newtheorem{propo}[theorem]{Proposition}

\newtheorem*{claim*}{Claim}

\newtheorem*{theorem*}{Theorem}
\newtheorem*{theorema}{Theorem A}

\theoremstyle{definition}

\theoremstyle{remark}
\newtheorem{rem}[theorem]{\bf Remark}

\numberwithin{equation}{section}

\newcommand{\brem}{\begin{rem}}
\newcommand{\erem}{\end{rem}}

\begin{document}

\title[Rectifiability and Riesz transforms]{New criteria for the rectifiability of Radon measures in terms of Riesz transforms}

\author{Xavier Tolsa}

\address{ICREA, Barcelona\\
Dept. de Matem\`atiques, Universitat Aut\`onoma de Barcelona \\
and Centre de Recerca Matem\`atica, Barcelona, Catalonia.}
\email{xavier.tolsa@uab.cat}

\thanks{Supported by the European Research Council (ERC) under the European Union's Horizon 2020 research and innovation programme (grant agreement 101018680). Also partially supported by MICIU (Spain) under the grant PID2024-160507NB-I00. 
}



\begin{abstract}
In this paper we explore the connection between quantitative rectifiability of measures and the $L^2$ boundedness of the codimension one Riesz transform. Among other things, we prove the following. Let $\mu$ be
 a Radon measure in $\R^{n+1}$ with growth of degree $n$ such that the $n$-dimensional Riesz transform $\RR_\mu$ is bounded in $L^2(\mu)$, and let $B_0\subset\R^{n+1}$ be a suitably doubling ball such that:
\begin{itemize}
\item[(i)]  There exists some (small) ball $B_1$ centered in $B_0$ with $\rad(B_1)\leq \delta_1\rad(B_0)$
such that, for some constant $\alpha>0$,
$$\frac{\mu(B_1)}{\rad(B_1)^n}\geq \alpha\,\frac{\mu(B_0)}{\rad(B_0)^n}.$$

\item[(ii)] For some $\ve>0$,
$$\int_{2B_0} |\RR\mu - m_{\mu,2B_0}(\RR\mu)|^2\,d\mu\leq \ve\,\bigg(\frac{\mu(B_0)}{\rad(B_0)^n}\bigg)^2\,\mu(B_0).$$ 
\end{itemize}
If $\delta_1$ is small enough, depending on $n$ and $\alpha$, and 
$\ve$ is small enough, then
there exists some uniformly $n$-rectifiable set $\Gamma$ and some $\tau>0$ such that
$\mu(\Gamma\cap 2B_0) \geq\tau\,\mu(B_0).$
\end {abstract}

\newcommand{\mih}[1]{\marginpar{\color{red} \scriptsize \textbf{Mih:} #1}}
\newcommand{\xavi}[1]{\marginpar{\color{blue} \scriptsize \textbf{Xavi:} #1}}

\maketitle

\section{Introduction}

In this paper we study the connection between the codimension one Riesz transform and the rectifiability of measures. Among other things,
for an arbitrary Radon measure $\mu$ in $\R^{n+1}$, we estimate the Jones-Wolff square function associated with the $\beta_2$ coefficients of the measure $\mu$ 
in terms of the $L^2(\mu)$ oscillation of the $n$-dimensional Riesz transform of $\mu$.
These results have applications to the study of harmonic measure and reflectionless measures for the Riesz transform.

We recall that a Borel measure $\mu$ in $\R^d$ is called $n$-rectifiable if there is a countable collection of  Lipschitz maps $g_i:\R^n\to\R^d$ such that
$\mu(\R^d\setminus \bigcup_i g_i(\R^n))=0$ and moreover $\mu$ is absolutely continuous with respect to the Hausdorff measure $\HH^n$.
A set $E\subset \R^d$ is $n$-rectifiable if the Hausdorff measure $\HH^n|_E$ is $n$-rectifiable. On the other hand, $E$ is called purely $n$-unrectifiable if $\HH^n(E\cap F) =0$ for any $n$-rectifiable set $F\subset\R^d$.

The ($n$-dimensional) Riesz transform of a signed Radon measure $\nu$ is defined by
$$\RR\nu (x) = \int \frac{x-y}{|x-y|^{n+1}}\,d\nu(y),$$
whenever the integral makes sense. 
Given $\ve>0$, the $\ve$-truncated Riesz transform of $\nu$ equals
$$\RR_\ve\nu (x) = \int_{|x-y|>\ve} \frac{x-y}{|x-y|^{n+1}}\,d\nu(y),$$
We also write
$$\RR_*\nu(x) = \sup_{\ve>0} |\RR_{\ve}\nu(x)|, \qquad  \pv\RR\nu(x) = \lim_{\ve>0} \RR_{\ve}\nu(x),$$
in case that the latter limit exists. Remark that, quite often, abusing notation we will
write $\RR\nu$ instead of $\pv\RR\nu$.

For $f\in L^1_{loc}(\mu)$ and a positive Radon measure $\mu$,
one writes $\RR_\mu f = \RR(f\mu)$ and $\RR_{\mu,\ve} f = \RR_\ve(f\mu)$. 
 We say that $\RR_\mu$ is bounded in $L^2(\mu)$ if
the operators $\RR_{\mu,\ve}$ are bounded uniformly in $L^2(\mu)$ uniformly on $\ve$, and then we denote
$$\|\RR_\mu\|_{L^2(\mu)\to L^2(\mu)} = \sup_{\ve>0} \|\RR_{\mu,\ve}\|_{L^2(\mu)\to L^2(\mu)}.$$

The study of the connection between quantitative rectifiability and the ·$L^2$ boundedness of the Riesz transform has been a subject of
intense research in the last decades because of the application to the Painlev\'e problem for bounded holomorphic and Lipschitz harmonic functions (see \cite{Tolsa-sem} and \cite{Volberg}).
In \cite{NToV}, Nazarov, the author, and Volberg showed that, given an $n$-Ahlfors regular set $E\subset\R^{n+1}$, the $L^2(\HH^n|_E)$ boundedness of the Riesz transform $\RR_{\HH^n|_E}$ implies the uniform $n$-rectifiability of $E$. See \rf{eqAD1} and \rf{eqUR1} for the definitions of
Ahlfors regularity and uniform rectifiability.
Without the Ahlfors regularity assumption, as shown in  \cite{NToV2}, for a set $E\subset\R^{n+1}$ with finite Hausdorff
measure $\HH^n$, the $L^2(\HH^n|_E)$ boundedness of $\RR_{\HH^n|_E}$ implies the $n$-rectifiability of $E$.
In the planar case, the same results had been proved previously by Mattila, Melnikov, and Verdera \cite{MMV} in the Ahlfors regular case, and
 by David and L\'eger \cite{Leger} in the non-Ahlfors regular one
using the connection between Menger curvature and the Cauchy kernel discovered by Melnikov \cite{Melnikov}.
From the previous results and a suitable $Tb$ type theorem, it turns out that a set $E$ with $\HH^n(E)<\infty$ is removable for Lipschitz harmonic functions if and only if
it is purely unrectifiable. In the plane this was shown by David and Mattila \cite{David-Mattila} (and by \cite{David-vitus} in the case of  holomorphic functions)  and by Nazarov, the author, and Volberg in higher dimensions \cite{NToV2}. 

More recently, in \cite{DT}, a geometric characterization of general compact non-removable sets, possibly with infinite measure $\HH^n$, was given in terms of the finiteness of a non-linear potential involving the $\beta_2$ coefficients of the measure. To describe this result in more detail, we introduce some notation.

For a Radon measure $\mu$ in $\R^d$, $x\in\R^d$, and $r>0$, we denote
$$\Theta_\mu^n(x,r) = \frac{\mu(B(x,r))}{r^n},$$
and for $1\leq p <\infty$,
$$\beta_{p,\mu}^n(x,r) = \inf_L \left(\frac1{r^n} \int_{B(x,r)} \left(\frac{\dist(y,L)}r \right)^p\,d\mu(y)\right)^{1/p},$$
where the infimum is taken over all $n$-planes $L\subset \R^d$. 
For a ball $B=B(x,r)$, we write $\Theta_\mu^n(B) = \Theta_\mu^n(x,r)$ and $\beta_{2,\mu}^n(B)=\beta_{2,\mu}^n(x,r)$.
In the case $d=n+1$, to shorten notation, we write $\Theta_\mu$ and $\beta_{p,\mu}$ instead of $\Theta_\mu^n$ and $\beta_{p,\mu}^n$.
One of the main results from \cite{DT} is the following.

\begin{theorema}\label{teoA}
Let $\mu$ be a Radon measure in $\R^{n+1}$ 
satisfying the growth condition
\begin{equation}\label{eqgrow00}
\mu(B(x,r))\leq \theta_0\,r^n\quad \mbox{ for all $x\in\supp\mu$ and all $r>0$}
\end{equation}
and such that $\|\RR_*\mu\|_{L^1(\mu)}<\infty$.
Then
\begin{equation}\label{eqbetawolff}
\int\!\!\int_0^\infty \beta_{2,\mu}(x,r)^2\,\Theta_\mu(x,r)\,\frac{dr}r\,d\mu(x)\leq C\,(\big\|\RR\mu\|_{L^2(\mu)}^2
+\theta_0^2\,\|\mu\|\big),
\end{equation}
where $C$ is an absolute constant.
\end{theorema}

We call the function $J_\mu(x) := \left(\int_0^\infty \beta_{2,\mu}(x,r)^2\,\Theta_\mu(x,r)\,\frac{dr}r\right)^{1/2}$ the Jones-Wolff potential (or Jones-Wolff square function) of $\mu$. As shown in \cite{DT}, from Theorem A it follows that a compact set $E\subset\R^{n+1}$ is non-removable for Lipschitz harmonic functions
if and only if it supports a non-zero measure $\mu$ satisfying the growth condition \rf{eqgrow00} whose potential $J_\mu$ is uniformly bounded in
$\R^{n+1}$.

The connection between quantitative rectifiability and the Riesz transform has also been applied recently to solve some long standing free boundary problems about harmonic measure. Indeed, relying on the results from \cite{NToV} and \cite{NToV2}, in
the work \cite{AHM3TV} it was shown that the mutual absolute continuity of harmonic measure and the Hausdorff measure $\HH^n$ on a subset $E\subset\pom$ implies the $n$-rectifiability of $E$. 
On the other hand, the works  \cite{AMT-cpam} and \cite{AMTV} solve the following two phase problem for harmonic measure: given two disjoint domains $\Omega_1,\Omega_2\subset \R^{n+1}$ with associated harmonic measures $\omega_1$ and $\omega_2$, if $\omega_1$ and $\omega_2$ are mutually absolutely continuous on a subset $E\subset\pom_1\cap\pom_2$, then there exists an $n$-rectifiable subset $F\subset E$ with full harmonic measure in $E$ such that both harmonic measures are mutually absolutely continuous with $\HH^n|_F$ on $E$. A fundamental ingredient of the proof is a criterion for the $n$-rectifiability of general Radon measures in terms of the $L^2$ oscillation of their Riesz transforms obtained by Girela-Sarri\'on and the author in \cite{Girela-Tolsa}.
In the plane the aforementioned two phase problem had been solved previously by Bishop in \cite{Bishop-arkiv}
by different arguments. 

In this work we extend the results from \cite{DT} and \cite{Girela-Tolsa} in different ways.  To state them we need some additional terminology. For a ball $B=B(x,r)$,
we denote
$$\PP_\mu(B)\equiv \PP_\mu(B(x,r)) := \sum_{j\geq0} 2^{-j}\,\Theta_\mu(2^jB).$$
We say that the ball $B$ is $\PP_\mu$-doubling if, for some fixed constant $C$,
$$\PP_\mu(B)\leq C\,\Theta_\mu(B).$$
We may say that $B$ is $(\PP_\mu,C)$-doubling if we want to specify the constant $C$.
We also set 
$$M_n\mu(x) = \sup_{r>0} \frac{\mu(B(x,r))}{r^n} = \sup_{r>0} \Theta_\mu(x,r)$$
and
$$\theta_\mu^{n,*}(x) = \limsup_{r\to0}\frac{\mu(B(x,r))}{r^n}.$$

The first result of this paper can be considered as local version of Theorem A:

\begin{theorem}\label{teoDTlocal00}
Let $\mu$ be a Radon measure in $\R^{n+1}$ and let $B_0\subset\R^{n+1}$ be a closed ball. Suppose that $M_n(\chi_{3B_0}\mu)\in L^2(\mu|_{2B_0})$ and 
$\RR_*(\chi_{3B_0}\mu)\in L^2(\mu|_{2B_0})$. 
Then we have
\begin{align}\label{eqgk812**}
\int_{B_0} |M_n(\mu|_{B_0})|^2\,&d\mu +  \int_{B_0}\int_0^{2\rad(B_0)} \beta_{2,\mu}(x,r)^2\,\Theta_\mu(x,r)\,\frac{dr}rd\mu(x)\\
 &\lesssim
\int_{2B_0} |\RR\mu - m_{\mu,2B_0}(\RR\mu)|^2\,d\mu + \PP_\mu(B_0)^2\mu(2B_0) + \int_{2B_0} \theta_\mu^{n,*}(x)^2\,d\mu(x),\nonumber
\end{align}
with the implicit constant depending on $n$.
\end{theorem}

\vv

A natural question motivated for possible applications to the study of harmonic measure and reflectionless measures is if on the right hand side
of \rf{eqgk812**} one can eliminate the terms $\PP_\mu(B_0)^2\mu(2B_0)$ or $\int_{2B_0} \theta_\mu^{n,*}(x)^2\,d\mu(x)$. 
Our main achievement in this paper is that, assuming  $B_0$ to be $\PP_\mu$-doubling and the density $\Theta_\mu(x,r)$ not to decay too fast
as $r\to0$ in some portion of $B_0$, the term the terms $\PP_\mu(B_0)^2\mu(2B_0)$ can
be eliminated. Indeed, this is controlled by the other two terms on the the right hand side. The precise result is the following.

\begin{theorem}\label{teomain}
Let $\mu$ be a Radon measure in $\R^{n+1}$ and let $B_0$ be a $(\PP_\mu,C_0)$-doubling ball. Suppose that $M_n(\chi_{3B_0}\mu)\in L^2(\mu|_{2B_0})$ and 
$\RR_*(\chi_{3B_0}\mu)\in L^2(\mu|_{2B_0})$. Suppose that there exists some ball $B_1$ centered in $B_0$ such that
$$\Theta_\mu(B_1)\geq \alpha\,\Theta_\mu(B_0)$$
for some $\alpha>0$, and such that, for some $\delta_0,\delta_1>0$,
$$\delta_0\,\rad(B_0)\leq \rad(B_1)\leq \delta_1\,\rad(B_0),$$
with $\delta_1\in (0,1/2)$ small enough, depending on $n$, $C_0$, and $\alpha$. Then there exists some constant $c_1>0$ depending on 
$n$, $\alpha$, and $\delta_0$ such that
$$c_1 \Theta_\mu(B_0)^2\,\mu(B_0)\leq \int_{2B_0} |\RR\mu - m_{\mu,2B_0}(\RR\mu)|^2\,d\mu + \int_{2B_0} \theta_\mu^{n,*}(x)^2\,d\mu(x).$$
\end{theorem}
\vv

As a corollary, putting Theorem \ref{teoDTlocal00} and Theorem \ref{teomain} together, we derive:

\begin{coro}\label{corotxulo}
Let $\mu$ be a Radon measure and $B_0$ a ball satisfying the assumptions in Theorem \ref{teomain}.
Then we have
\begin{align*}
\int_{B_0} |M_n(\mu|_{B_0})|^2\,d\mu +  \int_{B_0}&\int_0^{2\rad(B_0)} \beta_{2,\mu}(x,r)^2\,\Theta_\mu(x,r)\,\frac{dr}rd\mu(x)\\
 &\lesssim
\int_{2B_0} |\RR\mu - m_{\mu,2B_0}(\RR\mu)|^2\,d\mu + \int_{2B_0} \theta_\mu^{n,*}(x)^2\,d\mu(x),
\end{align*}
with the implicit constant depending on $n$, $\alpha$, and $\delta_0$.
\end{coro}

\vv

The next corollary deals with reflectionless measures. For our purposes, a reflectionless measure for the $n$-dimensional Riesz transform is a Radon measure $\mu$ in $\R^d$ having $n$-polynomial growth (i.e, satisfying \rf{eqgrow00}), such that $\RR_\mu$ is bounded in $L^2(\mu)$, and such that $\RR\mu$ is zero in a $\BMO(\mu)$ sense, i.e., $\RR\mu$ is constant $\mu$-a.e. In \cite{Jaye-Nazarov2}, 
Jaye and Nazarov have shown that if all the $n$-Ahlfors reflectionless measures for the Riesz transform are $n$-flat, i.e., of the form
$c\,\HH^n|_L$ for some $n$-plane $L$, then the 
the $n$-dimensional David-Semmes problem has a positive answer. That is, for any $n$-Ahlfors regular set $E\subset\R^d$, the
$L^2(\HH^n|_E)$ boundedness of $\RR_{\HH^n|_E}$ implies the uniform $n$-rectifiability of $E$, which is still an open question in $\R^d$ for $1<n<d$. 
However, up to now the structure of 
reflectionless measure for the Riesz transform is not well understood. It is not known if they are all $n$-flat even in the $n$-Ahlfors regular case with
$n=d-1$. For more information on reflectionless measures, see also \cite{Jaye-Nazarov1}.

We also remark that by now it is not known if there exist reflectionless measures $\mu$ such that the set $\{x\in\supp(\mu):\theta_\mu^{n,*}(x)=0\big\}$ has positive $\mu$-measure.
In this paper we will prove the following.

\begin{coro}\label{cororeflec}
Let $\mu$ be a Radon measure in $\R^{n+1}$ with polynomial $n$-growth which is reflectionless. There is some absolute constant
$\ve_n>0$
such that the following holds.
Suppose that $\dim_H(\mu)\leq n+\ve_n$. Then the set
$$\{x\in\supp(\mu):\theta_\mu^{n,*}(x)>0\big\}$$
is rectifiable and dense in $\supp(\mu)$.
\end{coro}

We recall that, for a Radon measure $\mu$ in $\R^d$, $\dim_H(\mu)$ stands for the Hausdorff dimension of $\mu$, defined by
$$\dim_H(\mu) = \inf\big\{\dim_H(E):\,\text{$E\subset\R^d$ Borel, $\mu(E^c)=0$}\big\}.$$
\vv

Another consequence of Theorem \ref{teomain} is the following enhanced version of a theorem from \cite{Girela-Tolsa}:

\begin{theorem}\label{teomain2}
Let $\mu$ be a Radon measure in $\R^{n+1}$ and let $B_0$ be a $(\PP_\mu,C_0)$-doubling ball. Suppose that 
the following conditions hold:
\begin{itemize}
\item[(a)] $\RR_{\mu|_{2B_0}}$ is bounded in $L^2(\mu|_{2B_0})$ with
$$\|\RR_\mu\|_{L^2(\mu|_{2B_0})\to L^2(\mu|_{2B_0})}\leq C_1\Theta_\mu(B_0)$$
and
$$\sup_{0<r\leq 4\rad(B_0)} \Theta_\mu(x,r) \leq C_1\Theta_\mu(B_0) \quad\mbox{ for all $x\in 2B_0$}.$$
\item[(b)] There exists some ball $B_1$ centered in $B_0$ such that, for some positive constants $\alpha$, $\delta_0$, $\delta_1$,
$$\Theta_\mu(B_1)\geq \alpha\,\Theta_\mu(B_0)$$
and
$$\delta_0\,\rad(B_0)\leq \rad(B_1)\leq \delta_1\,\rad(B_0).$$
\item[(c)] For some $\ve>0$,
$$\int_{2B_0} |\RR\mu - m_{\mu,2B_0}(\RR\mu)|^2\,d\mu\leq \ve\,\Theta_\mu(B_0)^2\,\mu(B_0).$$ 
\end{itemize}
Suppose that $\delta_1\in (0,1/2)$ is small enough, depending on $n$, $C_0$, $C_1$, and $\alpha$, and that
$\ve$ is small enough, depending on $C_0$, $C_1$, $\alpha$ and $\delta_0$. Then
there is a uniformly $n$-rectifiable set $\Gamma$ and some $\tau>0$ such that
$$\mu(\Gamma\cap 2B_0) \geq\tau\,\mu(B_0),$$
with $\tau$ and the uniform rectifiability  constant of $\Gamma$ depending on the above constants.
\end{theorem}

This theorem should be compared with Theorem 1.1 from \cite{Girela-Tolsa} (see also Corollary 3.2 from \cite{AMT-cpam} for an equivalent form closer to
the statement above). The main difference is that the condition (b) above is replaced by another asking $\beta_{1,\mu}(B_0)\leq \delta\,\Theta_\mu(B_0)$, with $\delta$ sufficiently
small. It is not difficult to check that the latter condition implies the assumption (b) of the theorem.
The main difference between the arguments to prove Theorem \ref{teomain2} and the ones in \cite{Girela-Tolsa} is that the ones in the latter
work are based on a delicate periodization of the measure $\mu$ in order to use later a variational argument leading to a contradiction. Instead, the arguments for the above Theorem \ref{teomain2} are based on Theorem \ref{teomain}, which in turn uses some techniques involving
harmonic measure (see for example, Lemma \ref{lemCDC} and Lemma \ref{lemmaxpr} below). 

We also remark that in Section \ref{secfinal} we
show a somewhat weaker variant of Theorem \ref{teomain2} which suffices for some applications to harmonic measure. This variant has the advantage that its proof does not rely on the lengthy arguments from \cite{DT}. 

Concerning the applications to harmonic measure, we mention that in a forthcoming work of Luis Lloret and the author \cite{Lloret-Tolsa}, the following result is obtained
by using Theorem \ref{teomain2} as an essential tool.

\begin{theorem*}
Let $\Omega\subset\R^{n+1}$ be an open set, and denote by $\omega$ the harmonic measure for $\Omega$ with pole at a fixed point in $\Omega$.
Let
$$Z=\big\{x\in\pom:\,\theta_{\omega}^{*,n}(x)=0\big\}.$$
Then, for $\omega$-a.e.\ $x\in Z$ it holds
\begin{equation}\label{eqosc}
\limsup_{r\to0}\frac{\sup_{0<s<r}\Theta_{\omega}(x,s)}{\Theta_{\omega}(x,r)} = +\infty.
\end{equation}
\end{theorem*}

Remark that, by the definition, $\Theta_{\omega}(x,r)\to0$ as
$r\to0$ for every $x\in Z$ with $Z\subset\pom$ as in the theorem. In a sense, \rf{eqosc} asserts that $\Theta_{\omega}(x,r)$ has arbitrarily large relative oscillations as it tends to $0$ (equivalently, the density $ \Theta_{\omega}(x,\cdot)$ fails to be almost-increasing near $r=0$). 
We recall that a deep result due to Wolff \cite{Wolff-plane}, in the complex plane, asserts that for any open set $\Omega\subset\R^2$, harmonic measure is concentrated in a set
of $\sigma$-finite length, which implies that $\omega(Z)=0$ in this case. In higher dimensions this is not the case, and in fact, one may have
$\dim_H(\omega)>n$ for $\Omega\subset\R^{n+1}$, by another celebrated result of Wolff \cite{Wolff-counterexample}.

\vv

\vv

\section{Preliminaries}\label{secprelim}

In the paper, constants denoted by $C$ or $c$ depend just on the dimension and perhaps other fixed
parameters. We will write $a\lesssim b$ if there is $C>0$ such that $a\leq Cb$ . We write $a\approx b$ if $a\lesssim b\lesssim a$.

\subsection{Measures, capacities, and the Riesz transform}

We denote the Lebesgue measure in $\R^d$ by $\LL^d$, although sometimes we also use the standard notation $dx$ or $dy$ to indicate integration
with respect to Lebesgue measure.

We say that a Radon measure $\mu$ in $\R^d$ has polynomial $n$-growth if there exists some constant $C$ such that
$$\mu(B(x,r))\leq C\,r^n\quad \mbox{ for all $x\in\R^d$, $r>0$.}$$
The measure $\mu$ is called  $n$-Ahlfors regular if
\begin{equation}\label{eqAD1}
C^{-1}r^n\leq \mu(B(x,r))\leq C r^n \quad \mbox{ for all $x\in\supp\mu$ and $0<r\leq \diam(\supp\mu)$.}
\end{equation}
On the other hand, $\mu$ is called uniformly $n$-rectifiable if it is $n$-Ahlfors regular and
there exist constants $\theta, M >0$ such that for all $x \in \supp\mu$ and all $0<r\leq \diam(\supp\mu)$ 
there is a Lipschitz mapping $g$ from the ball $B_n(0,r)$ in $\R^{n}$ to $\R^d$ with $\text{Lip}(g) \leq M$ such that
\begin{equation}\label{eqUR1}
\mu(B(x,r)\cap g(B_{n}(0,r)))\geq \theta r^{n}.
\end{equation}
The notion of uniform rectifiability is a quantitative version of $n$-rectifiability introduced by David and Semmes (see \cite{DS1} and
\cite{DS}) which has attracted much attention because its applications to harmonic analysis and PDE's in rough settings, among other things.
A set $E\subset\R^d$ is $n$-Ahlfors regular if $\HH^n|_E$ is $n$-Ahlfors regular, and it is uniformly $n$-rectifiable if 
$\HH^n|_E$ is uniformly $n$-rectifiable.

Given a set $F\subset\R^{d}$, the notation $M_+(F)$ stands for the
set of (positive) Radon measures supported on $F$, and $M_1(F)$ stands for the
set of Radon probability measures supported on $F$.

For $s>0$, the $s$-dimensional Hausdorff measure is denoted by $\HH^s$, and the $s$-dimensional Hausdorff
content by $\HH^s_\infty$. Recall that, for any set $E\subset \R^{n+1}$,
$$\HH^s_\infty(E) = \inf\Big\{\sum_i \diam(A_i)^s:E\subset\bigcup_i A_i\Big\}.$$
We say that $E$ is lower $s$-content regular if there exists some constant $c>0$ such that
$$\HH^s_\infty(B(x,r)\cap E)\geq c\,r^s\quad \text{ for all $x\in E$ and $0<r\leq\diam(E)$.}$$

Let $\EE$ denote the fundamental solution of the minus Laplacian in $\R^{n+1}$. That is, $$\EE_n(x) = \frac{c_n}{|x|^{n-1}},$$
for $n\geq2$, and 
$$\EE_1(x) = \frac1{2\pi}\,\log \frac1{|x|}$$
in the plane.
For a Radon measure $\mu$, we consider the potential defined by
$$U\mu(x) = \EE * \mu(x).$$
Given a set $F\subset\R^{n+1}$ 
we define the capacity $\capp(F)$ by the identity
\begin{equation}\label{eqicapimu}
	\capp(F) = \frac 1{\inf_{\mu\in M_1(F)} I(\mu)},
\end{equation}
where the infimum is taken over all {\em probability} measures $\mu$ supported on $F$ and $I(\mu)$ is the energy
\begin{equation*}\label{eqimu}
	I(\mu) = \iint \EE_n(x-y)\,d\mu(x)\,d\mu(y) = \int U\mu(x)\,d\mu(x).
\end{equation*}
For $n\geq 2$ and $F\subset \R^{n+1}$, $\capp(F)$ is the Newtonian capacity of $F$, and for $n=1$ and $F\subset\R^2$, 
$\capp(F)$ is the Wiener capacity of $F$. In the plane, the logarithmic capacity $\capp_L$, is defined by
$$\capp_L(F) = e^{-\frac{2\pi}{\capp(F)}}.$$

For $\ve>0$, we define the smooth $\ve$-truncated Riesz transform as follows. Fix a radial non-negative $C^\infty$ bump function 
$\vphi:\R^d\to\R$ which vanishes in $B(0,1)$ and equals $1$ in $\R^d\setminus B(0,2)$. Then for any signed measure $\nu$, we write
$$\wt \RR_\ve\nu(x)  = \int \vphi\Big(\frac{x-y}\ve\Big)\,\frac{x-y}{|x-y|^{n+1}}\,d\nu(y).$$
We define similarly $\wt \RR_{\mu,\ve} f$, for any positive measure $\mu$ and $f\in L^1_{loc}(\mu)$.
It is immediate to check that
$$|\RR_\ve \nu(x) - \wt\RR_\ve \nu(x)|\lesssim M_n\nu(x).$$
Then, for any measure $\mu$ with polynomial $n$-growth, the $L^2(\mu)$ boundedness of $\wt \RR_{\mu,\ve}$ is equivalent to the
the $L^2(\mu)$ boundedness of $\RR_{\mu,\ve}$. The advantage of $\wt \RR_{\mu,\ve}$ over $ \RR_{\mu,\ve}$ is that the kernel
of $\wt \RR_{\mu,\ve}$ is of Calder\'on-Zygmund type, uniformly on $\ve$.

\vv

\subsection{The David-Mattila lattice}

Next we introduce the dyadic lattice of cubes
with small boundaries of David-Mattila \cite{David-Mattila} associated with a Radon measure~$\mu$. 

\begin{theorem}[David, Mattila]
	\label{lemcubs}
	Let $\mu$ be a compactly supported Radon measure in $\R^{d}$.
	Consider two constants $C_0>1$ and $A_0>5000\,C_0$ and denote $E=\supp\mu$. 
	Then there exists a sequence of partitions of $E$ into
	Borel subsets $Q$, $Q\in \DD_{\mu,k}$, with the following properties:
	\begin{itemize}
		\item For each integer $k\geq0$, $E$ is the disjoint union of the ``cubes'' $Q$, $Q\in\DD_{\mu,k}$, and
		if $k<l$, $Q\in\DD_{\mu,l}$, and $R\in\DD_{\mu,k}$, then either $Q\cap R=\varnothing$ or else $Q\subset R$.
		\vv
		
		\item The general position of the cubes $Q$ can be described as follows. For each $k\geq0$ and each cube $Q\in\DD_{\mu,k}$, there is a ball $B(Q)=B(x_Q,r(Q))$ such that
		$$x_Q\in E, \qquad A_0^{-k}\leq r(Q)\leq C_0\,A_0^{-k},$$
		$$E\cap B(Q)\subset Q\subset E\cap 28\,B(Q)=E \cap B(x_Q,28r(Q)),$$
		and
		$$\mbox{the balls\, $5B(Q)$, $Q\in\DD_{\mu,k}$, are disjoint.}$$
		
		\vv
		\item The cubes $Q\in\DD_{\mu,k}$ have small boundaries. That is, for each $Q\in\DD_{\mu,k}$ and each
		integer $l\geq0$, set
		$$N_l^{ext}(Q)= \{x\in E\setminus Q:\,\dist(x,Q)< A_0^{-k-l}\},$$
		$$N_l^{int}(Q)= \{x\in Q:\,\dist(x,E\setminus Q)< A_0^{-k-l}\},$$
		and
		$$N_l(Q)= N_l^{ext}(Q) \cup N_l^{int}(Q).$$
		Then
		\begin{equation}\label{eqsmb2}
			\mu(N_l(Q))\leq (C^{-1}C_0^{-3d-1}A_0)^{-l}\,\mu(90B(Q)).
		\end{equation}
		\vv
		
		\item Denote by $\DD_{\mu,k}^{db}$ the family of cubes $Q\in\DD_{\mu,k}$ for which
		\begin{equation}\label{eqdob22}
			\mu(100B(Q))\leq C_0\,\mu(B(Q)).
		\end{equation}
		We have that $r(Q)=A_0^{-k}$ when $Q\in\DD_{\mu,k}\setminus \DD_{\mu,k}^{db}$
		and
		\begin{equation}\label{eqdob23}
			\mu(100B(Q))\leq C_0^{-l}\,\mu(100^{l+1}B(Q))\quad
			\mbox{for all $l\geq1$ with $100^l\leq C_0$ and $Q\in\DD_{\mu,k}\setminus \DD_{\mu,k}^{db}$.}
		\end{equation}
	\end{itemize}
\end{theorem}

\vv

\begin{rem}\label{rema00}
	We choose the constants $C_0$ and $A_0$ so that
	$$A_0 = C_0^{C(d)},$$
	where $C(d)$ depends
	just on $d$ and $C_0$ is big enough.
\end{rem}

We use the notation $\DD_\mu=\bigcup_{k\geq0}\DD_{\mu,k}$. Observe that the families $\DD_{\mu,k}$ are only defined for $k\geq0$. So the diameters of the cubes from $\DD_\mu$ are uniformly
bounded from above.
We set
$\ell(Q)= 56\,C_0\,A_0^{-k}$ and we call it the side length of $Q$. Notice that 
$$C_0^{-1}\ell(Q)\leq \diam(28B(Q))\leq\ell(Q).$$
Observe that $r(Q)\approx\diam(Q)\approx\ell(Q)$.
Also we call $x_Q$ the center of $Q$, and the cube $Q'\in \DD_{\mu,k-1}$ such that $Q'\supset Q$ the parent of $Q$.
We denote the family of cubes from $\DD_{\mu,k+1}$ which are contained in $Q$ by $\Ch(Q)$, and we call their elements children or sons of $Q$.
We set
$B_Q=28 B(Q)=B(x_Q,28\,r(Q))$, so that 
$$E\cap \tfrac1{28}B_Q\subset Q\subset B_Q\subset B(x_Q,\ell(Q)/2).$$

For a given $\gamma\in(0,1)$, let $A_0$ be big enough so that the constant $C^{-1}C_0^{-3d-1}A_0$ in 
\rf{eqsmb2} satisfies 
$$C^{-1}C_0^{-3d-1}A_0>A_0^{\gamma}>10.$$
Then we deduce that, for all $0<\lambda\leq1$,
\begin{align}\label{eqfk490}
	\mu\bigl(\{x\in Q:\dist(x,E\setminus Q)\leq \lambda\,\ell(Q)\}\bigr) + 
	\mu\bigl(\bigl\{x\in 3.5B_Q\setminus Q:\dist&(x,Q)\leq \lambda\,\ell(Q)\}\bigr)\\
	&\leq_\gamma
	c\,\lambda^{\gamma}\,\mu(3.5B_Q).\nonumber
\end{align}

We denote
$\DD_\mu^{db}=\bigcup_{k\geq0}\DD_{\mu,k}^{db}$.
Note that, in particular, from \rf{eqdob22} it follows that
\begin{equation}\label{eqdob*}
	\mu(3B_{Q})\leq \mu(100B(Q))\leq C_0\,\mu(Q)\qquad\mbox{if $Q\in\DD_\mu^{db}.$}
\end{equation}
For this reason we will call the cubes from $\DD_\mu^{db}$ doubling. 
Given $Q\in\DD_\mu$, we denote by $\DD_\mu(Q)$
the family of cubes from $\DD_\mu$ which are contained in $Q$. Analogously,
we write $\DD_\mu^{db}(Q) = \DD^{db}_\mu\cap\DD_\mu(Q)$.

As shown in \cite[Lemma 5.28]{David-Mattila}, every cube $R\in\DD_\mu$ can be covered $\mu$-a.e.\
by a family of doubling cubes from $\DD^{db}_\mu$.

The following result is proved in \cite[Lemma 5.31]{David-Mattila}.
\vv

\begin{lemma}\label{lemcad22}
	Let $k\le j$, $R\in\DD_{\mu,k}$ and $Q\in\DD_{\mu,j}\cap\DD_{\mu}(R)$ be a cube such that all the intermediate cubes $S$,
	$Q\subsetneq S\subsetneq R$ are non-doubling (i.e.\ belong to $\DD_\mu\setminus \DD_\mu^{db}$).
	Suppose that the constants $A_0$ and $C_0$ in Lemma \ref{lemcubs} are
	chosen as in Remark \ref{rema00}. 
	Then
	\begin{equation}\label{eqdk88}
		\mu(100B(Q))\leq A_0^{-10d(j-k-1)}\mu(100B(R)).
	\end{equation}
\end{lemma}
\vv

Recall that, given a ball (or an arbitrary set) $B\subset \R^{n+1}$,  its $n$-dimensional density (with respect to $\mu$) is defined by
$\Theta_\mu(B)= \frac{\mu(B)}{\rad(B)^n}.$
From the previous lemma it follows:

\vv
\begin{lemma}\label{lemcad23}
	Let $Q,R\in\DD_\mu$ be as in Lemma \ref{lemcad22}.
	Then
	$$\Theta_\mu(100B(Q))\leq (C_0A_0)^{n+1}\,A_0^{-9d(j-k-1)}\,\Theta_\mu(100B(R))$$
	and
	$$\sum_{S\in\DD_\mu:Q\subset S\subset R}\Theta_\mu(100B(S))\leq c\,\Theta_\mu(100B(R)),$$
	with $c$ depending on $C_0$ and $A_0$.
\end{lemma}

For the easy proof, see
\cite[Lemma 4.4]{Tolsa-memo}, for example.

We need some additional notation.
Given $Q\in\DD_\mu$ and $\lambda>1$, we denote by $\lambda Q$ the union of cubes $P$ from the same
generation as $Q$ such that $\dist(x_Q,P)\leq \lambda \,\ell(Q)$. Notice that
\begin{equation}\label{eqlambq12}
	\lambda Q\subset B(x_Q,(\lambda+\tfrac12)\ell(Q)).
\end{equation}
Also, we let
$$\DD_\mu(\lambda Q)=\{P\in\DD_\mu:P\subset \lambda Q,\,\ell(P)\leq \ell(Q)\}.$$

The following result is from \cite[Lemma 2.6]{DT}:

\begin{lemma}\label{lemDMimproved}
	Let $\mu$ be a compactly supported Radon measure in $\R^{d}$.
	Assume that $\mu$ has polynomial growth of degree $n$ and let $\gamma\in(0,1)$. The lattice $\DD_\mu$ from Theorem
	\ref{lemcubs} can be constructed so that the following holds for all
	all $Q\in\DD_{\mu}$:
	\begin{align*}
		\int_{2B_Q\setminus Q}\left(\int_Q \frac1{|x-y|^n}\,d\mu(y)\right)^2 d\mu(x)\, 
		+ &\int_{Q}\left(\int_{2B_Q\setminus Q} \frac1{|x-y|^n}\,d\mu(y)\right)^2 d\mu(x)\\
		&\leq C(\gamma)\sum_{P\in\DD_\mu: P\subset 2Q} \left(\frac{\ell(P)}{\ell(Q)}\right)^\gamma\Theta_\mu(2B_P)^2\mu(P).
	\end{align*}
\end{lemma}

This estimate can be seen as an enhanced version of the small boundaries condition. Indeed, if the measure $\mu$ were $n$-Ahlfors regular,
then the right hand side above would be comparable to $\theta_\mu(2B_Q)^2\mu(Q)$ and the result would be an easy consequence of the
small boundaries condition.

\subsection{$\PP$-doubling cubes, energies, and martingale decompositions in $L^2(\mu)$}

For $Q\in\DD_\mu$,
we denote
$$\PP_\mu(Q) = \sum_{R\in\DD_\mu:R\supset Q} \frac{\ell(Q)}{\ell(R)^{n+1}} \,\mu(2B_R).$$
We say that a cube $Q$ is $\PP$-doubling if
$$\PP_\mu(Q) \leq C_d\,\frac{\mu(2B_Q)}{\ell(Q)^n},$$
for  $C_d =4A_0^n$. 
Notice that
$$\PP_\mu(Q) \approx_{C_0} \sum_{R\in\DD_\mu:R\supset Q} \frac{\ell(Q)}{\ell(R)} \,\Theta_\mu(2B_R).$$
and thus saying that $Q$ is $\PP$-doubling implies that 
$$\sum_{R\in\DD_\mu:R\supset Q} \frac{\ell(Q)}{\ell(R)} \,\Theta_\mu(2B_R)\leq C_d'\,\Theta_\mu(2B_Q)$$
for some $C_d'$ depending on $C_d$. It is also immediate to check that then $2B_Q$ is $\PP_\mu$-doubling, with another constant depending on $C_d$. Conversely, the latter condition implies that $Q$ is $\PP$-doubling with another constant $C_d$ depending on $C_d'$.
The following is proven in \cite[Lemma 3.1]{DT}.

\begin{lemma}\label{lempois00}
	Suppose that $C_0$ and $A_0$ are chosen suitably. If $Q\in\DD_\mu$ is $\PP$-doubling, then $Q\in\DD_\mu^{db}$.
	Also, any cube $R\in\DD_\mu$ such that $R\cap 2Q\neq\varnothing$ and $\ell(R)=A_0\ell(Q)$ belongs to $\DD_\mu^{db}$.
\end{lemma}

For technical reasons, in \cite{DT} the following discrete version of the density $\Theta_\mu$ is introduced. Given  $Q\in\DD_\mu$, we let
$$\Theta_\mu(Q) = A_0^{kn} \quad \mbox{ if\, $\dfrac{\mu(2B_Q)}{\ell(Q)^n}\in [A_0^{kn},A_0^{(k+1)n})$}.$$
Clearly, $\Theta_\mu(Q)\approx \Theta_\mu(2B_Q)$. This discrete version of the density is appropriate for some of the stopping time arguments
involved in \cite{DT}.

Given $Q\in\DD_\mu$ and $k\geq1$, we denote by $\hd^k(Q)$ the family of maximal cubes $P\in\DD_\mu$ satisfying
\begin{equation}\label{a0tilde}
	\ell(P)<\ell(Q), \qquad \Theta_\mu(P)\geq  A_0^{kn}\Theta_\mu(Q).
\end{equation}
For given $\lambda\geq1$ and $Q\in\DD_\mu$, we consider the energies
$$\EE(\lambda Q) = \sum_{P\in\DD_\mu(\lambda Q)} \left(\frac{\ell(P)}{\ell(Q)}\right)^{3/4}\Theta_\mu(P)^2\,\mu(P)$$
and
$$\EE_\infty(\lambda Q) = \sup_{k\geq1}
\sum_{P\in\hd^k(Q)\cap\DD_\mu(\lambda Q)} \left(\frac{\ell(P)}{\ell(Q)}\right)^{\!1/2}\Theta_\mu(P)^2\,\mu(P).$$
These energies play an important role in some of the estimates involving the Riesz transform in \cite{DT}.
By Lemma 3.4 from \cite{DT}, for every $Q\in\DD_\mu$ we have
$$\EE(\lambda Q)  \lesssim\EE_\infty(\lambda Q).$$
For some fixed $M_0\gg1$, we say that $Q\in\DD_\mu$ is dominated from below, and we write $Q\in\DB(M_0)$, if
\begin{equation}\label{eqdefBB} 
\EE_\infty(9 Q)\geq M_0^2\,\Theta_\mu(Q)^2\mu(9Q).
\end{equation}

For $f\in L^2(\mu)$ and $Q\in\DD_\mu$ we define
\begin{equation}\label{eqdq1}
	\Delta_Q f=\sum_{S\in\Ch(Q)}m_{\mu,S}(f)\chi_S-m_{\mu,Q}(f)\chi_Q,
\end{equation}
where $m_{\mu,S}(f)$ stands for the average of $f$ on $S$ with respect to $\mu$.
Then, for any cube $R\in\DD_\mu$, we have the orthogonal expansion 
$$\chi_{R} \bigl(f - m_{\mu,R}(f)\bigr) = \sum_{Q\in\DD_\mu(R)}\Delta_Q f,$$
in the $L^2(\mu)$-sense, so that
$$\|\chi_{R} \bigl(f - m_{\mu,R}(f)\|_{L^2(\mu)}^2 = \sum_{Q\in\DD_\mu(R)}\|\Delta_Q f\|_{L^2(\mu)}^2.$$

\vv


\section{The approximation lemma}

In this section we introduce some approximating measures and we will prove a technical result which will be necessary in the next sections.

Let $\mu$ be a Radon measure in $\R^{n+1}$ and let $B_0\subset\R^{n+1}$ be a closed ball. Suppose that $M_n\mu\in L^2(\mu|_{2B_0})$ and 
$\RR_*\mu\in L^2(\mu|_{2B_0})$. 
We will define two approximating measures for $\mu$. To this end, we consider the dyadic lattice $\DD_\mu$ associated with $\mu$  from Theorem \ref{lemcubs} and an integer $k$ such that $\ell(Q)\ll \rad(B_0)$ for the cubes $Q\in\DD_{\mu,k}$. 
For each $Q\in\DD_{\mu,k}$, let $S_Q$ be a ball in $x_Q$, the center of $Q$, with radius
$\rad(B(Q))/10$. Set $\Gamma_Q= \partial S_Q$. Denote by $I_k$ the family of cubes from $\DD_{\mu,k}$ that intersect $2B_0$.
Then we let
\begin{equation}\label{eqaprox98}
\sigma_k = \chi_{\R^{n+1}\setminus 2B_0}\,\mu + \sum_{Q\in I_k} \frac{\mu(Q\cap 2B_0)}{\HH^n(\Gamma_Q)}\,\HH^n|_{\Gamma_Q}
\end{equation}
and
\begin{equation}\label{eqaprox99}
\wt \sigma_k = \chi_{\R^{n+1}\setminus 2B_0}\,\mu + \sum_{Q\in I_{B_0}} \frac{\mu(Q\cap 2B_0)}{\HH^{n+1}(S_Q)}\,\HH^{n+1}|_{S_Q}.
\end{equation}
\vv

\begin{lemma}[Approximation Lemma]\label{lemmapprox}
Let $\mu$ be a Radon measure in $\R^{n+1}$ and let $B_0\subset\R^{n+1}$ be a closed ball. Suppose that $M_n\mu\in L^2(\mu|_{2B_0})$ and 
$\RR_*\mu\in L^2(\mu|_{2B_0})$. Let $\sigma_k$ and $\wt \sigma_k$ be as above.
Then the functions $M_n\sigma_k$, $\RR_*\sigma_k$ belong to $L^2(\sigma_k|_{2B_0})$, and  $M_n\wt \sigma_k$, $\RR_*\wt \sigma_k$ belong to $L^2(\wt \sigma_k|_{2B_0})$. Moreover,
\begin{align}\label{eqaprox1}
\limsup_{k\to \infty} \int_{2B_0} |\RR\sigma_k - m_{\sigma_k,2B_0}& (\RR\sigma_k)|^2\,d\sigma_k\\
& \leq 2
\int_{2B_0} |\RR\mu - m_{\mu,2B_0}(\RR\mu)|^2\,d\mu + C\int_{2B_0} \theta_\mu^{n,*}(x)^2\,d\mu(x).\nonumber
\end{align}
The same estimate holds with $\sigma_k$ replaced by $\wt \sigma_k$.
\end{lemma}

\begin{proof}
The fact that $M_n\sigma_k\in L^2(\sigma_k|_{2B_0})$ and $M_n\wt\sigma_k\in L^2(\sigma_k|_{2B_0})$ follow easily by checking that
for any $x\in 2B_0 \cap S_Q$, with $Q\in I_k$,
\begin{equation}\label{eqaprox0}
M_n\sigma_k(x) + M_n\wt\sigma_k(x) \lesssim \inf_{y\in Q\cap2B_0} M_n\mu(y).
\end{equation}
We leave the details for the reader.

Next we show that $\RR_*\sigma_k\in L^2(\sigma_k|_{2B_0})$, and \rf{eqaprox1}.
We will use the smooth truncated Riesz transform $\RE$. We take first $\ve\geq A_0^{-k}$. Then for $x\in \Gamma_Q$,
 with $Q\in I_k$, 
 we set
\begin{align*}
\RE\sigma_k(x) - \RE\mu(x) &=  \sum_{P\in I_k} \big(\RE(\chi_{\Gamma_P}\sigma_k)(x) - \RE(\chi_{P\cap 2B_0}\mu)(x)\big)\\
& =  \sum_{P\in I_k}\int K_\ve(x-z)\,d(\sigma_k|_{\Gamma_P} -\mu|_{P\cap2B_0})(z),
\end{align*}
where $K_\ve$ is the kernel of $\RE$.  
Using that $\sigma_k(\Gamma_P) = \mu(P\cap2B_0)$ and denoting $\ell_k= A_0^{-k}$, we deduce
\begin{align*}
|\RE\sigma_k(x) - \RE\mu(x)| & \leq  
\sum_{P\in I_k}\int |K_\ve(x-z) - K_\ve(x-z_P)|\,d|\sigma_k|_{\Gamma_P} -\mu|_{P\cap2B_0}|(z)\\
& \lesssim  
\sum_{P\in I_k}\int \frac{\ell_k}{(\ve + |x-z|)^{n+1}} \,d|\sigma_k|_{\Gamma_P} -\mu|_{P\cap2B_0}|(z)\\
& \lesssim  
\sum_{P\in I_k}\frac{\ell_k\,\mu(P)}{(\ve + \dist(x,P))^{n+1}} \\
 &\lesssim  \frac{\ell_k\,\mu(B(x,C\ell_k))}{\ve^{n+1}} +
\sum_{j\geq 0}\sum_{\substack{P\in I_k:\\ 2^j\ell_k\leq \dist(x,P)<2^{j+1}\ell_k}}\!\!\frac{\ell_k\,\mu(P)}{\dist(x,P)^{n+1}}.
\end{align*}

Recalling that $\ve\geq\ell_k$, the first summand can be bounded above by $\Theta_\mu (B(x,C\ell_k))$.
We estimate the last sum as follows
$$
\sum_{j\geq 0}\sum_{\substack{P\in I_k:\\ 2^j\ell_k\leq \dist(x,P)<2^{j+1}\ell_k}}\!\!\frac{\ell_k\,\mu(P)}{\dist(x,P)^{n+1}} 
\lesssim \sum_{j\geq 0}\frac{\ell_k}{2^j \ell(P)} \,\frac{\mu(B(x_P,2^j\ell(P))}{(2^j \ell(P))^n} \lesssim \PP_\mu(Q).
$$
Thus,
$$|\RE\sigma_k(x) - \RE\mu(x)|\lesssim \PP_\mu(Q).$$

On the other hand, for $x\in \Gamma_Q$ and $y\in Q$, we have
$$|\RE\mu(x) - \RE\mu(y)|\lesssim \int |K_\ve(x-z) - K_\ve(y-z)| \,d\mu(z) \lesssim \int\frac{\ell_k}{(\ve + |x-z|)^{n+1}} \,d\mu(z)
\lesssim 
 \PP_\mu(Q).$$
Then, by the triangle inequality,
$$|\RE\sigma_k(x) - \RE\mu(y)|\lesssim \PP_\mu(Q).$$
So, for any constant $\rho\in\R^{n+1}$, we have
\begin{equation}\label{eqaprox2}
|\RE\sigma_k(x)-\rho|\leq \inf_{y\in Q\cap2B_0} |\RE\mu(y)-\rho| + C\PP_\mu(Q)\quad \mbox{ for all $x\in \Gamma_Q$, $\ve\geq\ell_k$.}
\end{equation}

Using that 
$$\sup_{\ve>0}|\RE\mu(y)|\leq \RR_*\mu(y) + M_n\mu(y)$$
and taking $\rho=0$ in \rf{eqaprox2}, we obtain
\begin{equation}\label{eqaprox4} 
\sup_{\ve\geq\ell_k} |\RE\sigma_k(x)|\leq  \inf_{y\in Q\cap2B_0}(\RR_* \mu(y)+ CM_n\mu(y)) + C\PP_\mu(Q)\leq \inf_{y\in Q\cap2B_0}(\RR_* \mu(y)+ M_n\mu(y)).
\end{equation}
On the other hand, from the construction of $\sigma_k$, for $0<\ve\leq \ell_k$, using the smoothness of $\Gamma_Q$, it follows easily that
\begin{equation}\label{eqaprox5}
|\RE \sigma_k(x) - \wt\RR_{\ell_k}\sigma_k(x)|\lesssim \Theta_\mu (B(x,C\ell_k))\quad \mbox{ for all $x\in \Gamma_Q$,}
\end{equation}
for some fixed $C>1$. 
Together with \rf{eqaprox4} and \rf{eqaprox0}, this yields
$$ \RR_* \sigma_k(x) \leq  \sup_{\ve >0} |\RE\sigma_k(x)|+CM_n\sigma_k(x) \leq
 \inf_{y\in Q\cap2B_0}(\RR_* \mu(y)+ CM_n\mu(y)).
$$
Squaring and integrating on $2B_0$, we obtain
\begin{align*}
\int_{2B_0} |\RR_* \sigma_k|^2\,d\sigma_k & \leq \sum_{Q\in I_k} \sup_{x\in \Gamma_Q}|\RR_* \sigma_k(x)|^2\,\sigma_k(Q)\\
&\leq \sum_{Q\in I_k} \inf_{y\in Q\cap 2B_0}|(\RR_* \mu(y) + CM_n\mu(y))^2\,\mu(Q\cap 2B_0)\\
&\leq 2\int_{2B_0} |\RR_* \mu|^2\,d\mu + C\int_{2B_0} |M_n \mu|^2\,d\mu<\infty,
\end{align*}
which proves that $\RR_*\sigma_k\in L^2(\sigma_k)$.

Next we will show \rf{eqaprox1}. Taking $\rho = m_{\mu,2B_0} (\RR\mu)$ in \rf{eqaprox2}, 
and combining this with \rf{eqaprox5}
we obtain, for all $x\in \Gamma_Q$,
$$|\RR\sigma_k(x)-\rho|\leq |\wt\RR_{\ell_k}\sigma_k(x)- \rho| +\Theta_\mu (B(x,C\ell_k))\leq
\inf_{y\in Q\cap2B_0} |\wt\RR_{\ell_k}\mu(y)-\rho| + C\PP_\mu(Q).
$$
Squaring and integrating on $2B_0$ and taking into account that $\PP_\mu(Q)\approx \PP_\mu(B(x,\ell_k))$ for all $x\in Q\in I_k$, we deduce
\begin{align*}
\int_{2B_0} |\RR\sigma_k - m_{\sigma_k,2B_0} (\RR\sigma_k)|^2\,d\sigma_k &\leq 
\int_{2B_0}|\RR\sigma_k-\rho|^2\,d\sigma_k \\
&\leq 2 \int_{2B_0}|\wt\RR_{\ell_k}\mu-\rho|^2\,d\mu + C\int_{2B_0}\PP_\mu(B(x,\ell_k))^2\,d\mu(x).
\end{align*}
Since $\RR_*\mu\in L^2(\mu|_{2B_0})$, by dominated convergence, it follows that
$$\lim_{k\to\infty} \int_{2B_0}|\wt\RR_{\ell_k}\mu-\rho|^2\,d\mu = \int_{2B_0}|\RR\mu(x)-\rho|^2\,d\mu=
\int_{2B_0}|\RR\mu-m_{\mu,2B_0}(\RR\mu)|^2\,d\mu.$$
Hence, to prove \rf{eqaprox1}, it suffices to show that
$$\limsup_{k\to\infty} 
\int_{2B_0}\PP_\mu(B(x,\ell_k))^2\,d\mu(x) \lesssim \int_{2B_0}\theta_\mu^{n,*}(x)^2\,d\mu(x).$$
It is easy to check that if $M_n\mu(x)<\infty$, then
$$\limsup_{k\to\infty}\PP_\mu(B(x,\ell_k)) \lesssim \theta_\mu^{n,*}(x).$$
By dominated convergence, then
\begin{align*}
\limsup_{k\to\infty} 
\int_{2B_0}\PP_\mu(B(x,\ell_k))^2\,d\mu(x) & \leq \lim_{k\to\infty} 
\int_{2B_0}\sup_{j\geq k}\PP_\mu(B(x,\ell_j))^2\,d\mu(x) \\ 
&=
\int_{2B_0}\limsup_{k\to\infty}\PP_\mu(B(x,\ell_j))^2\,d\mu(x) \lesssim \int_{2B_0}\theta_\mu^{n,*}(x)^2\,d\mu(x),
\end{align*}
as wished. This concludes the proof of \rf{eqaprox1}.

The proof \rf{eqaprox1} for $\wt\sigma_k$ is essentially the same as the one above for $\sigma_k$ and we skip it.
\end{proof}

\vv

\vv

\section{The proof of Theorem \ref{teoDTlocal00}}

The first step for the proof of Theorem \ref{teoDTlocal00} is the following.

\begin{propo}\label{propoDT0}
Let $\mu$ be a finite Radon measure in $\R^{n+1}$ and let $S_0\in\DD_\mu$ be a $\PP$-doubling cube from the associated modified David-Mattila lattice.
 There exists some absolute constant $A>2$ such that if 
  $M_n\mu\in L^2(\mu|_{A S_0})$,
$\RR_*\mu\in L^2(\mu|_{AS_0})$, and for $\mu$-a.e.\ $x\in AS_0$ there exists a sequence of $\PP$-doubling cubes $P_k\in\DD_\mu$ with $x\in P_k$ and $\ell(P_k)\to0$,
 then we have
\begin{multline*}
\int_{S_0} |M_n(\chi_{S_0}\mu)|^2\,d\mu + \int_{S_0}\int_0^{\ell(S_0)} \beta_{2,\mu}(x,r)^2\,\Theta_\mu(x,r)\,\frac{dr}r\,d\mu(x)\\
 \lesssim
\int_{AS_0} |\RR\mu - m_{\mu,AS_0}(\RR\mu)|^2\,d\mu + \Theta_\mu(S_0)^2\mu(S_0) + \int_{AS_0} \theta_\mu^{n,*}(x)^2\,d\mu(x),
\end{multline*}
with the implicit constant depending only on $n$.
\end{propo}

\begin{proof}
The arguments are very similar to the ones in \cite{DT}. We only sketch the main differences and we use the same notation as in \cite{DT}.

We consider a corona decomposition in terms of a family $\ttt(S_0)\equiv \ttt$, defined as in Section 7 in \cite{DT}, with the difference that now $S_0$ does not coincide with the whole $\supp\mu$, as in \cite{DT}. In this way, by construction, all the cubes from $\ttt(S_0)$ are contained in $S_0$ and 
$$\DD_\mu(S_0)= \bigcup_{R\in\ttt(S_0)}\tree(R).$$
We recall that for a family of cubes $\FF\subset\DD_\mu$, in
\cite{DT} the notation $\sigma(\FF)$ stands for
$$\sigma(\FF) = \sum_{Q\in\FF}\Theta_\mu(Q)^2\mu(Q).$$
By \cite[Lemma 7.3]{DT},
$$\sigma(\ttt(S_0)) \lesssim \sigma(\ttt(S_0)\cap \MDW) + \sigma(S_0)$$
(in \cite{DT} it is written $\theta_0^2\,\|\mu\|$ in place of $\sigma(S_0)$).
Then Lemma 8.2 in \cite{DT} holds replacing $\theta_0^2\,\|\mu\|$ by $\sigma(S_0)$. 

At the beginning of Section 9, devoted to the proof of Main Lemma 7.1, for $Q\in \Trc_k(R)$, $R\in\sL$, in the case when 
$\mu(Z(Q))>\ve_Z\,\mu(Q)$, it is written that
\begin{equation}\label{eqnou87}
\sigma(\HD_1(e(Q)))\leq \theta_0^2\,\mu(e(Q)) \lesssim \ve_Z^{-1}\,\theta_0^2\,\mu(Z(Q)).
\end{equation}
Instead, now we write the more precise estimate
\begin{equation}\label{eqnou88}
\sigma(\HD_1(e(Q))) \lesssim \ve_Z^{-1} \Theta_\mu(\HD_1(Q))^2\,\mu(Z(Q)) \lesssim_{\ve_Z,\Lambda,\delta_0} \int_{Z(Q)}
\theta_\mu^{n,*}(x)^2\,d\mu(x).
\end{equation}
The last estimate follows from the fact that if $x\in Z(Q)$, then $x\not\in P$ for any $P\in\End(Q)$. 
By the assumption that for $\mu$-a.e.\ there exists a sequence of $\PP$-doubling cubes $P_k\in\DD_\mu$ with $x\in P_k$ with side length tending to $0$, it follows that the cubes from 
$\End(Q)$ cover $\mu$-almost all $\bigcup_{P\in \sss_2(Q)}P$. Then we deduce that, for  $x\in Z(Q)$, 
any cube $P\subset \DD_\mu$  such that $x\in P$ with $\ell(P)<\ell(Q)$ satisfies
$$\PP_\mu(P)\geq \delta_0\,\Theta_\mu(Q).$$
This easily implies that $\theta_\mu^{n,*}(x)\gtrsim \delta_0\,\Theta_\mu(Q)$.
Using this estimate, similarly to (9.3) in \cite{DT}, one gets
\begin{align*}
\sigma(\ttt(S_0)) & \lesssim  B^{5/4}\,\Lambda^2
\sum_{R\in\sL}\,  \sum_{k\geq0} B^{-k/2}\sum_{Q\in\Trc_k(R)}
\|\Delta_{\wt \TT(e'(Q))} \RR\mu\|_{L^2(\mu)}^2\\
&\quad+ B^{5/4}\,
\Lambda^{\frac1{3n}}
\sum_{R\in\sL}\,  \sum_{k\geq0} B^{-k/2}\sum_{Q\in\Trc_k(R)}
\sum_{P\in\DB:P\sim\TT(e'(Q))} \!\!\EE_\infty(9P)\nonumber\\
&\quad+ C(\ve_Z,\Lambda,\delta_0)\,B^{5/4}\,
 \sum_{R\in\sL}\,  \sum_{k\geq0} B^{-k/2}\sum_{Q\in\Trc_k(R)}\int_{Z(Q)}
\theta_\mu^{n,*}(x)^2\,d\mu(x) + \sigma(S_0),\nonumber
\end{align*}
where $\DB=\DB(M_0)$ is as in \rf{eqdefBB}, with some sufficiently large $M_0$.
Remark that $e'(Q),Z(Q)\subset 4S_0$ for any $Q\in\Trc_k(R)$ with $R\in\sL$, by construction. Also, it is easy to check that the cubes $P\in\DB$ appearing above are contained in $A^{1/2}S_0$, for some absolute constant $A>10$.
Then, arguing as in \cite[Lemma 9.2]{DT}, it follows that
\begin{align*}
\sigma(\ttt(S_0)) & \lesssim_{\Lambda} \sum_{Q\in\DD_\mu:Q\subset 4S_0} \| \Delta_Q\RR\mu\|_{L^2(\mu)}^2
 + \sum_{P\in \DB:P\subset A^{1/2}S_0} \EE_\infty(9P)\\
 &\quad + \sigma(S_0) + 
 C(\ve_Z,\Lambda,\delta_0)
 \int_{4S_0}
\theta_\mu^{n,*}(x)^2\,d\mu(x).
\end{align*}
By the mutual orthogonality of the functions $\Delta_Q\RR\mu$ and standard estimates, it follows that
$$\sum_{Q\in\DD_\mu:Q\subset 4S_0} \| \Delta_Q\RR\mu\|_{L^2(\mu)}^2\leq \int_{4S_0} |\RR\mu - m_{\mu,4S_0}(\RR\mu)|^2\,d\mu .$$
Putting all together,
\begin{align*}
\sum_{R\in\ttt(S_0)} \Theta_\mu(R)^2\mu(R) & \lesssim_{\Lambda} \int_{4S_0} |\RR\mu - m_{\mu,4S_0}(\RR\mu)|^2\,d\mu
 + \sum_{P\in \DB:P\subset A^{1/2}S_0} \EE_\infty(9P)\\
 &\quad + \sigma(S_0) + 
 C(\ve_Z,\Lambda,\delta_0)
 \int_{4S_0}
\theta_\mu^{n,*}(x)^2\,d\mu(x).
\end{align*}
Further, by Lemma 10.1 from \cite{DT}, the same estimate holds with the left hand side replaced by $\sum_{Q\in\DD_\mu(S_0)}\beta_{\mu,2}(2B_Q)^2
\Theta_\mu(Q)\mu(Q)$. Also, by the stopping conditions in the definition of the family $\ttt(S_0)$, it is immediate to check that
$$\int_{S_0} |M_n(\chi_{S_0}\mu)|^2\,d\mu\lesssim_\Lambda \sum_{R\in\ttt(S_0)} \Theta_\mu(R)^2\mu(R).$$
So we deduce that
\begin{align}\label{eqtoq94}
&\int_{S_0} |M_n(\chi_{S_0}\mu)|^2\,d\mu + \sum_{Q\in\DD_\mu(S_0)}\beta_{\mu,2}(2B_Q)^2
\Theta_\mu(Q)\mu(Q)\\
 &\lesssim
\int_{4S_0} \!\!|\RR\mu - m_{\mu,4S_0}(\RR\mu)|^2d\mu \,+\, \Theta_\mu(S_0)^2\mu(S_0)\, + \int_{4S_0}\!\! \theta_\mu^{n,*}(x)^2\,d\mu(x)+\!\! \sum_{P\in \DB:P\subset A^{1/2}S_0}\!\! \EE_\infty(9P),\nonumber
\end{align}
with the implicit constant depending on the parameters $\Lambda,\delta_0,\ve_0$.

To conclude the proof of the proposition, it remains to deal with the term involving the cubes from the family $\DB$ in the last estimate.
This is done as in Sections 11-13 from \cite{DT}, replacing the family $\DB$ there by $\DB_{S_0}:=\{P\in \DB:P\subset A^{1/2}S_0\}$.
Then, as in \cite[Section 13]{DT}, we get
$$\sum_{Q\in\DB_{S_0}} \EE_\infty(9Q) \lesssim \Lambda^{\frac{-1}{2n}}(\log\Lambda)^2 
\sum_{R\in\sL(\GDF)}\,
\sum_{k\geq0} B^{-k/2} \sum_{Q\in\Trc_k(R)\cap\Ty}\sigma(\HD_1(e(Q))),$$
with $\GDF$ defined as in \cite{DT}, using $\DB_{S_0}$ in place of $\DB$. 
To bound $\sigma(\HD_1(e(Q)))$ for $Q\in \Trc_k(R)$, $R\in\sL(\GDF)$, in the case when 
$\mu(Z(Q))>\ve_Z\,\mu(Q)$, we use again the estimate \rf{eqnou88} instead of \rf{eqnou87}. Using this modification, arguing as in \cite{DT}
it follows that
\begin{equation}\label{eqDB**}
\sum_{Q\in\DB_{S_0}} \EE_\infty(9Q)\lesssim_\Lambda \sum_{Q\in\DD_\mu:Q\subset AS_0} \| \Delta_Q\RR\mu\|_{L^2(\mu)}^2 + 
\int_{4S_0}\!\! \theta_\mu^{n,*}(x)^2\,d\mu(x).
\end{equation}
Together with \rf{eqtoq94}, this proves the proposition.
\end{proof}
\vv

\begin{proof}[\bf Proof of Theorem \ref{teoDTlocal00}]
Clearly, we may suppose that $\mu(B_0)>0$. First we assume that $\mu$ is a finite measure and that for $\mu$-a.e.\ $x\in 2B_0$ there exists a sequence of $\PP$-doubling cubes $P_k\in\DD_\mu$ with $x\in P_k$ and $\ell(P_k)\to0$. From the finiteness of $\mu$, it follows that 
$\RR_*(\chi_{\R^{n+1}\setminus 3B_0}\mu)\in L^\infty(\mu|_{2B_0})$. Together with the fact that $\RR_*(\chi_{3B_0}\mu)\in L^2(\mu|_{2B_0})$,
this ensures that $\RR_*\mu\in L^2(\mu|_{2B_0})$. Similarly, from the fact that $M_n(\chi_{3B_0}\mu)\in L^2(\mu|_{2B_0})$, we deduce that
$M_n\mu\in L^2(\mu|_{2B_0})$.

Let $I_0\subset \DD_\mu$ be a family of maximal cubes from $Q\in\DD_\mu$ such that $2Q\subset \frac32B_0$ and $Q\cap B_0\neq\varnothing$. Then we have $\ell(R)\approx \rad(B_0)$ for each $R\in I_0$.
We denote $R_0 = \bigcup_{R\in I_0} R$ and 
we consider the measure
$$\eta =  \mu|_{R_0}.$$
By Proposition \ref{propoDT0} applied to $\eta$ with $S_0=R_0$ and the fact that $\mu(2B_0)\geq\eta(R_0)$, we deduce that
\begin{align}\label{eqal8ed}
\int_{B_0} |M_n&(\mu|_{B_0})|^2\,d\mu  + 
 \int_{B_0}\int_0^{2\rad(B_0)} \beta_{2,\mu}(x,r)^2\,\Theta_\mu(x,r)\,\frac{dr}rd\mu(x)\\
& \lesssim \Theta_{\mu}(R_0)^2\,\mu(R_0) + \int_{R_0} |M_n\eta|^2\,d\eta + \int_{R_0}\int_0^{2\rad(B_0)} \beta_{2,\eta}(x,r)^2\,\Theta_{\eta}(x,r)\,\frac{dr}rd\eta(x)
\nonumber\\
 & \lesssim \Theta_{\mu}(R_0)^2\mu(R_0) + \int_{R_0} |\RR\eta|^2\,d\eta + \int_{R_0} \theta_{\mu}^{n,*}(x)^2\,d\mu(x).\nonumber
\end{align}
Next, using the antisymmetry of the Riesz kernel we can write
\begin{align*}
\int_{R_0} |\RR\eta|^2\,d\eta  &= \int_{R_0} |\RR(\chi_{R_0}\mu) - m_{\mu,B_0}(\RR(\chi_{R_0}\mu))|^2\,d\mu \\  
&\lesssim \int_{R_0}\! |\RR\mu - m_{\mu,R_0}(\RR\mu)|^2\,d\mu
+ \int_{R_0}\! |\RR(\chi_{2B_0\setminus R_0}\mu) - m_{\mu,R_0}(\RR(\chi_{2B_0\setminus R_0}\mu))|^2\,d\mu\\
& + \int_{R_0} |\RR(\chi_{\R^{n+1}\setminus 2B_0}\mu) - m_{\mu,R_0}(\RR(\chi_{\R^{n+1}\setminus2B_0}\mu))|^2\,d\mu.
\end{align*}
Using that, for $x,y\in R_0$,
$$|\RR(\chi_{\R^{n+1}\setminus 2B_0}\mu)(x) - \RR(\chi_{\R^{n+1}\setminus2B_0}\mu)(y)|\lesssim \PP_\mu(B_0),$$
we deduce that
$$\int_{R_0} |\RR\eta|^2\,d\eta\lesssim 
\int_{R_0}\! |\RR\mu - m_{\mu,R_0}(\RR\mu)|^2\,d\mu + \int_{R_0} |\RR(\chi_{2B_0\setminus R_0}\mu)|^2\,d\mu + \PP_\mu(B_0)^2\,\mu(R_0).$$ 
From \rf{eqal8ed} and the last inequalities, we derive
\begin{align}\label{eqal8ed1}
\int_{B_0} |M_n(\mu&|_{B_0})|^2\,d\mu  + \int_{B_0} \beta_{2,\mu}(x,r)^2\,\Theta_\mu(x,r)\,d\mu(x)\\
& \lesssim 
\int_{R_0}\! |\RR\mu - m_{\mu,R_0}(\RR\mu)|^2\,d\mu  + \PP_\mu(B_0)^2\,\mu(2B_0)
+ \int_{R_0} \theta_{\mu}^{n,*}(x)^2\,d\mu(x)\nonumber\\
&\quad + \int_{R_0} |\RR(\chi_{2B_0\setminus R_0}\mu)|^2\,d\mu.\nonumber
\end{align}

Our objective is to bound the last integral on the right hand side of \rf{eqal8ed1} in terms of the first three summands on the right hand side of \rf{eqal8ed1}.
To this end, we apply Lemma \ref{lemDMimproved}:
\begin{align}\label{eqal8ed2}
\int_{R_0} |\RR(\chi_{2B_0\setminus R_0}\mu)|^2\,d\mu&\lesssim
 \Theta_\mu(2B_0)^2\,\mu(R_0)
+ \sum_{Q\in I_0}\int_Q|\RR(\chi_{2B_Q\setminus Q}\mu)|^2\,d\mu \\
& \lesssim
 \Theta_\mu(2B_0)^2\,\mu(R_0) + \sum_{Q\in I_0}\sum_{P\subset 2Q}\bigg(\frac{\ell(P)}{\ell(Q)}\bigg)^{9/10}\Theta_\mu(2B_P)^2\,\mu(P)
\nonumber \\
 & \lesssim 
  \Theta_\mu(2B_0)^2\,\mu(R_0) + \sum_{P\in\DD_\mu:P\subset \frac32 B_0}\bigg(\frac{\ell(P)}{\rad(B_0)}\bigg)^{9/10}\Theta_\mu(2B_P)^2\,\mu(P),
\nonumber
\end{align}
where we took into account that $2Q\subset\frac32B_0$ for every $Q\in I_0$.
Next we consider the family $I_\PP\subset\DD_\mu$  of the maximal $\PP$-doubling cubes contained in $\frac32 B_0$ which have side length 
at most $\gamma\,\rad(B_0)$, where $\gamma<1/100$ will be fixed below. Notice that this family covers $\mu$-almost all $\frac32B_0$. We denote by $J_0$ the family of the cubes from $\DD_\mu$ which are contained in $\frac32 B_0$ are not contained in any cube from $I_\PP$. 
Then we split
\begin{align}\label{eqfj58}
\sum_{P\in\DD_\mu:P\subset \frac32 B_0}&\bigg(\frac{\ell(P)}{\rad(B_0)}\bigg)^{9/10}\Theta_\mu(2B_P)^2\,\mu(P)\\
& \leq
\sum_{P\in J_0}\bigg(\frac{\ell(P)}{\rad(B_0)}\bigg)^{9/10}\Theta_\mu(2B_P)^2\,\mu(P)+
\sum_{Q\in I_\PP}\sum_{P\subset Q} \bigg(\frac{\ell(P)}{\rad(B_0)}\bigg)^{9/10}\Theta_\mu(2B_P)^2\,\mu(P).\nonumber
\end{align}
To estimate the first sum on the right hand side, given $P\in J_0$, let $P'\supset P$ be a maximal cube with side length at most $\gamma\,\rad(B_0)$.
Since there are no $\PP$-doubling cubes $S\in\DD_\mu$ such that $P\subsetneq S\subset P'$, by \cite[Lemma 3.3]{DT}, we have
\begin{equation}\label{eqdecdens5f}
\Theta_\mu(2B_P)\leq \PP_\mu(P)\lesssim \PP_\mu(P') \lesssim_\gamma \PP_\mu(B_0).
\end{equation}
Therefore,
\begin{equation}\label{eqJ0}
\sum_{P\in J_0}\bigg(\frac{\ell(P)}{\rad(B_0)}\bigg)^{9/10}\!\Theta_\mu(2B_P)^2\,\mu(P)\lesssim_\gamma \!\!
\sum_{P\subset \frac32B_0}\bigg(\frac{\ell(P)}{\rad(B_0)}\bigg)^{9/10}\PP_\mu(B_0)^2\,\mu(P)\lesssim_\gamma \PP_\mu(B_0)^2\,\mu(2B_0).
\end{equation} 

Concerning the last sum on the right hand side of \rf{eqfj58}, we have
\begin{align}\label{eqdk291}
\sum_{Q\in I_\PP}\sum_{P\subset Q} \bigg(\frac{\ell(P)}{\rad(B_0)}\bigg)^{9/10}\Theta_\mu(2B_P)^2\,\mu(P)
& \leq
\sum_{Q\in I_\PP} \bigg(\frac{\ell(Q)}{\rad(B_0)}\bigg)^{3/4}
\sum_{P\subset Q} \bigg(\frac{\ell(P)}{\ell(Q)}\bigg)^{3/4}\Theta_\mu(2B_P)^2\,\mu(P)\\
 =&\sum_{Q\in I_\PP} \bigg(\frac{\ell(Q)}{\rad(B_0)}\bigg)^{3/4}\,\EE(Q)\lesssim\sum_{Q\in I_\PP} \bigg(\frac{\ell(Q)}{\rad(B_0)}\bigg)^{3/4}\,\EE_\infty(Q).\nonumber
\end{align}
If $Q\in I_\PP$ is not dominated from below (i.e., $Q\not\in \DB(M_0)$), by \rf{eqdecdens5f} we have
then 
$$\EE_\infty(Q)\lesssim \Theta_\mu(Q)^2\,\mu(Q)\lesssim \PP_\mu(B_0)^2\,\mu(Q.$$
On the other hand, if $Q\in \DB(M_0)$, by the Main Proposition 3.7 from \cite{DT}, by arguments such as the ones in \rf{eqDB**}, for each $\PP$-doubling cube it holds
$$
\EE_\infty(Q)\lesssim \int_{AQ} |\RR\mu - m_{\mu,AQ}(\RR\mu)|^2 \,d\mu + 
\int_{AQ} \theta_\mu^{n,*}(x)^2\,d\mu(x).$$
So in any case, for $Q\in I_\PP$ we have
\begin{align*}
\EE_\infty(Q)&\lesssim \Theta_\mu(Q)^2\,\mu(Q) + 
\int_{AQ} |\RR\mu - m_{\mu,AQ}(\RR\mu)|^2 \,d\mu + 
\int_{AQ} \theta_\mu^{n,*}(x)^2\,d\mu(x)\\
& \lesssim 
\PP_\mu(B_0)^2\,\mu(Q) + 
\int_{AQ} |\RR\mu - m_{\mu,2B_0}(\RR\mu)|^2 \,d\mu + 
\int_{AQ} \theta_\mu^{n,*}(x)^2\,d\mu(x).
\end{align*}
Plugging this estimate into \rf{eqdk291}, we get
\begin{align*}
&\sum_{Q\in I_\PP}\sum_{P\subset Q} \bigg(\frac{\ell(P)}{\rad(B_0)}\bigg)^{9/10}\Theta_\mu(2B_P)^2\,\mu(P)
\\
& \leq
\sum_{Q\in I_\PP} \bigg(\frac{\ell(Q)}{\rad(B_0)}\bigg)^{3/4}\,
\bigg(\PP_\mu(B_0)^2\,\mu(Q) + 
\int_{AQ} |\RR\mu - m_{\mu,2B_0}(\RR\mu)|^2 \,d\mu + 
\int_{AQ} \theta_\mu^{n,*}(x)^2\,d\mu(x)\bigg).
\end{align*}
Taking $\gamma$ small enough (depending on $A$), we have $AQ\subset 2B_0$ for any cube $Q\subset \frac32B_0$ with 
$\ell(Q)\leq\gamma\,\rad(B_0)$, and so in particular for the cubes $Q\in I_\PP$. Then, using also the finite superposition
of the ``cubes'' $AQ$ from a fixed generation, we get
\begin{align*}
&\sum_{Q\in I_\PP} \bigg(\frac{\ell(Q)}{\rad(B_0)}\bigg)^{3/4}\!
\int_{AQ} |\RR\mu - m_{\mu,2B_0}(\RR\mu)|^2 \,d\mu \\
& \quad\leq \!
\sum_{Q\subset \frac32B_0:AQ\subset 2B_0}\bigg(\frac{\ell(Q)}{\rad(B_0)}\bigg)^{3/4}
\int_{AQ} |\RR\mu - m_{\mu,2B_0}(\RR\mu)|^2 \,d\mu \lesssim 
\int_{2B_0} |\RR\mu - m_{\mu,2B_0}(\RR\mu)|^2 \,d\mu.
\end{align*}
An analogous argument shows that
$$\sum_{Q\in I_\PP} \bigg(\frac{\ell(Q)}{\rad(B_0)}\bigg)^{3/4} 
\int_{AQ} \theta_\mu^{n,*}(x)^2\,d\mu(x)\lesssim\int_{2B_0} \theta_\mu^{n,*}(x)^2\,d\mu(x).$$
Gathering the last estimates with \rf{eqJ0}, we derive
\begin{align*}
\sum_{P\in\DD_\mu:P\subset \frac32 B_0}&\bigg(\frac{\ell(P)}{\rad(B_0)}\bigg)^{9/10}\!\Theta_\mu(2B_P)^2\,\mu(P)
\\
&\lesssim_\gamma 
\PP_\mu(B_0)^2\mu(2B_0) + \int_{2B_0} |\RR\mu - m_{\mu,2B_0}(\RR\mu)|^2 \,d\mu +
\int_{2B_0} \theta_\mu^{n,*}(x)^2\,d\mu(x).
\end{align*}
Therefore,
$$
\int_{R_0} \!|\RR(\chi_{2B_0\setminus R_0}\mu)|^2\,d\mu\lesssim_\gamma \!
\PP_\mu(B_0)^2\mu(2B_0) + \int_{2B_0} \!|\RR\mu - m_{\mu,2B_0}(\RR\mu)|^2 \,d\mu +
\int_{2B_0}\!\! \theta_\mu^{n,*}(x)^2\,d\mu(x).$$
Together with \rf{eqal8ed1}, this concludes the proof of the theorem under the assumption that
for $\mu$-a.e.\ $x\in 2B_0$ there exists a sequence of $\PP$-doubling cubes $P_k\in\DD_\mu$ with $x\in P_k$ and $\ell(P_k)\to0$.
\vv

Consider now a general {\em finite} Radon measure $\mu$ satisfying the assumptions of the theorem. We will reduce this situation to the case above. 
To this end, for each $k$ large enough, we consider the measure
 $\sigma_k$ constructed in \rf{eqaprox98}. 
It is immediate to check that $\sigma_k$ is such that for $\sigma_k$-a.e.\ $x\in 2B_0$ there exists a sequence of $\PP$-doubling cubes (with respect to $\sigma_k$) $P_k\in\DD_{\sigma_k}$ with $x\in P_k$ and $\ell(P_k)\to0$.
So \rf{eqal8ed} holds with $\mu$ replaced by $\sigma_k$.

The assumption that $M_n\mu\in L^2(\mu|_{2B_0})$ and $\RR_*\mu\in L^2(\mu|_{2B_0})$ ensures that
\begin{align*}
\int_{B_0} |M_n&(\mu|_{B_0})|^2\,d\mu + \int_{B_0}\int_0^{2\rad(B_0)} \beta_{2,\mu}(x,r)^2\,\Theta_\mu(x,r)\,\frac{dr}rd\mu(x)\\
& \leq \liminf_{k\to\infty}\bigg(
\int_{B_0} |M_n(\sigma_k|_{B_0})|^2\,d\sigma_k + \int_{B_0}\int_0^{2\rad(B_0)} \beta_{2,\sigma_k}(x,r)^2\,\Theta_{\sigma_k}(x,r)\,\frac{dr}rd\sigma_k(x)\bigg),
\end{align*}
and by Lemma \ref{lemmapprox}, also that 
\begin{align*}
\limsup_{k\to\infty}\bigg(\int_{2B_0} &|\RR\sigma_k - m_{\sigma_k,2B_0} (\RR\sigma_k)|^2\,d\sigma_k + \PP_{\sigma_k}(B_0)^2\sigma_k(2B_0) + \int_{2B_0} \theta_{\sigma_k}^{n,*}(x)^2\,d\sigma_k(x)\bigg)\\
& \lesssim 
\int_{2B_0} |\RR\mu - m_{\mu,2B_0}(\RR\mu)|^2\,d\mu + \PP_\mu(B_0)^2\mu(2B_0) + \int_{2B_0} \theta_\mu^{n,*}(x)^2\,d\mu(x).
\end{align*}
The last two inequalities together with the validity of \rf{eqal8ed}
 for $\sigma_k$ imply the theorem for general finite measures.

\vv

Next, let $\mu$ be a general  Radon measure. Writing $\mu_0 = \chi_{3B_0}\mu$, we get
\begin{align}\label{eqgk812**-**}
 \int_{B_0} &|M_n(\mu|_{B_0})|^2\,d\mu +  \int_{B_0}\int_0^{2\rad(B_0)} \beta_{2,\mu}(x,r)^2\,\Theta_\mu(x,r)\,\frac{dr}rd\mu(x)\\
& =
\int_{B_0} |M_n(\mu_0|_{B_0})|^2\,d\mu_0 +  \int_{B_0}\int_0^{2\rad(B_0)} \beta_{2,\mu_0}(x,r)^2\,\Theta_{\mu_0}(x,r)\,\frac{dr}rd\mu_0(x)
\nonumber\\
 &\lesssim
\int_{2B_0} |\RR\mu_0 - m_{\mu_0,2B_0}(\RR\mu_0)|^2\,d\mu_0 + \PP_{\mu_0}(B_0)^2\mu_0(2B_0) + \int_{2B_0} \theta_{\mu_0}^{n,*}(x)^2\,d\mu_0(x)
\nonumber\\
& \leq
\int_{2B_0} |\RR\mu_0 - m_{\mu,2B_0}(\RR\mu_0)|^2\,d\mu + \PP_{\mu}(B_0)^2\mu(2B_0) + \int_{2B_0} \theta_{\mu}^{n,*}(x)^2\,d\mu(x).
\nonumber
\end{align}

Denoting $\mu^0 = \mu-\mu_0$, we deduce
\begin{align*}
\int_{2B_0} |\RR\mu_0 - m_{\mu,2B_0}(\RR\mu_0)|^2\,d\mu  
& \leq
2\int_{2B_0} |\RR\mu - m_{\mu,2B_0}(\RR\mu)|^2\,d\mu\\
& \quad+ 
2 \int_{2B_0} |\RR\mu^0 - m_{\mu,2B_0}(\RR\mu^0)|^2\,d\mu.
\end{align*}
By the smoothness properties of the Riesz kernel and standard arguments, using that $\mu^0$ vanishes in $3B_0$, for any $x,y\in 2B_0$
we have 
$$|\RR\mu^0(x) - \RR\mu^0(y)|\lesssim \PP_{\mu}(B_0).$$
This implies that
$$\int_{2B_0} |\RR\mu^0 - m_{\mu,2B_0}(\RR\mu^0)|^2\,d\mu\lesssim \PP_{\mu}(B_0)^2\,\mu(2B_0).$$
Therefore,
$$\int_{2B_0} |\RR\mu_0 - m_{\mu,2B_0}(\RR\mu_0)|^2\,d\mu \lesssim 
\int_{2B_0} |\RR\mu - m_{\mu,2B_0}(\RR\mu)|^2\,d\mu + \PP_{\mu}(B_0)^2\,\mu(2B_0).$$
Plugging this estimate into \rf{eqgk812**-**}, the theorem follows.
\end{proof}
\vv

\begin{rem}\label{rem2}
In the proof of Theorem \ref{teoDTlocal00} we have shown that
if $M_n\mu\in L^2(\mu|_{2B_0})$ and 
$\RR_*\mu\in L^2(\mu|_{2B_0})$, then we have
$$
\int_{R_0} \!|\RR(\chi_{2B_0\setminus R_0}\mu)|^2\,d\mu\lesssim_\gamma \!
\PP_\mu(B_0)^2\mu(2B_0) + \int_{2B_0} \!|\RR\mu - m_{\mu,2B_0}(\RR\mu)|^2 \,d\mu +
\int_{2B_0}\!\! \theta_\mu^{n,*}(x)^2\,d\mu(x),$$
under the assumption that for $\mu$-a.e.\ $x\in 2B_0$ there exists a sequence of $\PP$-doubling cubes $P_k\in\DD_\mu$ with $x\in P_k$ and $\ell(P_k)\to0$. By an approximation argument similar to the one at the end of the proof of that theorem, it turns out that the 
same estimate holds for arbitrary Radon measures $\mu$ such that $M_n(\chi_{3B_0}\mu)\in L^2(\mu|_{2B_0})$ and 
$\RR_*(\chi_{3B_0}\mu)\in L^2(\mu|_{2B_0})$.
\end{rem}

\vv


\section{The main lemma}

Let $\mu$ be a Radon measure in $\R^{n+1}$ and $B$ a ball. We say that $B$ is $(\mu,C_2)$-good if $\mu(B)>0$ and
$$\int_{B} M_n(\chi_B \mu)\,d\mu \leq C_2 \,\Theta_\mu(B)\mu(B).$$

The next lemma is one of the main technical steps for the proof of Theorem \ref{teomain}. It deals with the particular case when 
$\mu$ is absolutely continuous with respect to Lebesgue measure in $2B_0$. In the lemma we consider a ball $B_0$ and a set $R_0$ as in the proof of Theorem \ref{teoDTlocal00}. That is, we let
 $I_0\subset \DD_\mu$ be a family of maximal cubes from $Q\in\DD_\mu$ such that $2Q\subset \frac32B_0$ and $Q\cap B_0\neq\varnothing$. Then 
 we denote $R_0 = \bigcup_{R\in I_0} R$, and we have $\ell(R)\approx \rad(B_0)$ for each $R\in I_0$.

\begin{mlemma}\label{lemaclau}
Let $\mu$ be a Radon measure in $\R^{n+1}$ and let $B_0$ be a $(\PP_\mu,C_0)$-doubling ball. Suppose that there exists some $h\in L^\infty(\mu)$
such that $\mu = h\,\LL^{n+1}$ in $2B_0$ and that
 $M_n(\chi_{3B_0}\mu)\in L^2(\mu|_{2B_0})$ and $\RR_*(\chi_{3B_0}\mu)\in L^2(\mu|_{2B_0})$. Suppose that there exists some ball $B_1$ centered in $B_0$ such that
$$\Theta_\mu(B_1)\geq \alpha\,\Theta_\mu(B_0)$$
for some $\alpha>0$, 
and such that, for some $\delta_0,\delta_1>0$,
$$\delta_0\,\rad(B_0)\leq \rad(B_1)\leq \delta_1\,\rad(B_0),$$
with $\delta_1\in (0,1/2)$ small enough, depending on $n$, $C_0$, and $\alpha$.
Let $N= \lfloor\log_2 (\rad(B_0)/\rad(B_1)\rfloor$. Suppose that, for some $\eta>0$, at least $\eta N$ balls from the family $\{2^jB_1\}_{0\leq j\leq N}$ are $(\mu,C_2)$-good and that, for some $\tau>0$, 
\begin{equation}\label{eqvar110}
\int _{2B_0\setminus R_0}\int_{R_0} \frac1{|x-y|^n}\,d\mu(x)d\mu(y)\leq \tau
\Theta_\mu(B_0)\mu(2B_0).
\end{equation}
Then, if $\delta_1$ is  small enough, depending on $n$, $C_0$, $\alpha$, $\tau$, $C_2$, and $\eta$, we have
$$\int_{R_0} |\RR\mu - m_{\mu,R_0}(\RR\mu)|^2\,d\mu \geq c_1 \Theta_\mu(B_0)^2\,\mu(B_0),$$
for some constant $c_1>0$ depending on 
$n$, $\alpha$, $\delta_0$, and $C_2$, but not on $\|h\|_{L^\infty(\mu)}$.
\end{mlemma}

\vv


To prove this result, first we show the following.

\begin{lemma}\label{lemfrostman}
If $B\subset \R^{n+1}$ is $(\mu,C_2)$-good, then 
$$\HH^n_\infty(B\cap\supp\mu) \geq c\,C_2^{-1}\,\rad(B)^n.$$
\end{lemma}

\begin{proof}
Since $B$ is $(\mu,C_2)$-good, by Chebychev,
$$\mu\big(\big\{x\in B:  M_n(\chi_B \mu)>2C_2\,\Theta_\mu(B)\big\}\big) \leq \frac1{2C_2\,\Theta_\mu(B)} \int_B  M_n(\chi_B \mu)\,d\mu
\leq \frac12\,\mu(B).$$
Hence there is a compact set $E\subset B$ such that $\mu(E)\geq \mu(B)/4$ and $M_n(\chi_B \mu)(x)\leq 2C_2\,\Theta_\mu(B)$ for all $x\in E$.
So the measure 
$$\sigma:= \frac1{2C_2\,\Theta_\mu(B)}\,\mu|_E$$ 
satisfies $M_n\sigma(x)\leq 1$ for all $x\in\supp\sigma$ and 
$$\sigma(E) = \frac1{2C_2\,\Theta_\mu(B)}\,\mu(E) \geq \frac1{8C_2\,\Theta_\mu(B)}\,\mu(B) = \frac{\rad(B)^n}{8C_2}.$$
By Frostman's lemma, this implies that
$$\HH^n_\infty(B\cap\supp\mu)\geq \HH^n_\infty(E)\gtrsim \sigma(E) \geq \frac{\rad(B)^n}{8C_2}.$$
\end{proof}
\vv

From now on, in this section we assume that we are under the assumptions of  Lemma \ref{lemaclau}.
\vv

\begin{lemma}[Variational Lemma]\label{lemvar}
Let $\mu$ and $B_0$ satisfy the assumption of Main Lemma \ref{lemaclau} and let $R_0$ be as above. Denote
$c_{R_0}=m_{\mu,R_0}(\RR\mu)$.
Suppose that, for some $p\in (1,2]$ and some $\lambda\in (0,1)$,
\begin{equation}\label{eqvar1}
\int_{R_0} |\RR\mu - c_{R_0}|^p\,d\mu \leq \lambda \Theta_\mu(B_1)^p\,\mu(B_1).
\end{equation}
Then there exists a function $a\in L^\infty(\mu|_{R_0})$, with $0\leq a\leq 4$, such that the measure
$$\nu = \mu|_{\R^{n+1}\setminus R_0} + a\,\mu|_{R_0}$$ satisfies:
\begin{itemize}
\item[(a)] $\nu(B_1)\geq \frac14\,\mu(B_1)$. 
\item[(b)] At least $\eta N/2$ balls from the family $\{2^jB_1\}_{0\leq j\leq N}$ are $(\nu,1024 \eta^{-1} C_2)$-good, for 
 $N= \lfloor\log_2 (\rad(B_0)/\rad(B_1)\rfloor$. 
\item[(c)] For all $x\in \frac54 B_0\cap\supp\nu$, it holds
\begin{equation}\label{eqvar48}
|\RR\nu(x) - c_{R_0}|^p + p\,\RR_\nu^*(\chi_{R_0}|\RR\nu - c_{R_0}|^{p-2}(\RR\nu - c_{R_0}))(x) \leq 16\lambda\,\Theta_\mu(B_1)^p.
\end{equation}
\end{itemize}
\end{lemma}

Remark that the left hand side of \rf{eqvar48} may be negative.

\begin{proof}
We denote 
$$\sigma_p(B_1)= \Theta_\mu(B_1)^p\mu(B_1).$$
For any non-negative function $a\in L^\infty(\mu|_{R_0})$, we denote
$$\nu_a = \mu|_{\R^{n+1}\setminus R_0} + a\,\mu|_{R_0}$$
and we consider the functional
$$F(a) = \int_{R_0} |\RR\nu_a - c_{R_0}|^p\,d\nu_a + 
\lambda\,\sigma_p(B_1)\bigg(\|a\|_\infty^p + \frac{\mu(B_1)}{\nu_a(B_1)} +  \frac1N\sum_{k=1}^N \frac{\mu(2^kB_1)}{\nu_a(2^kB_1)}\bigg).$$
Observe that
$$F(1) = \int_{R_0} |\RR\mu - c_{R_0}|^p\,d\mu + 3\lambda\,\sigma_p(B_1)\leq 4 \lambda\sigma_p(B_1),$$
by the assumption \rf{eqvar1}. This implies that the infimum of $F$ over the non-negative function $a\in L^\infty(\mu|_{R_0})$ is attained
for $\|a\|_\infty\leq 4$. Then, by standard arguments, one deduces that there exists a minimizing function $a\in L^\infty(\mu|_{R_0})$ for the functional
$F:L^\infty(\mu|_{R_0})\to [0,\infty]$. The minimizer, which we also denote by $a$, is a non-negative function such that $\|a\|_\infty\leq 4$.

The fact that $F(a)\leq F(1)\leq 4\lambda\sigma_p(B_1)$ also implies that
\begin{equation}\label{eqvar3}
\frac{\mu(B_1)}{\nu_a(B_1)} \leq 4\quad \text{ and }\quad \frac1N\sum_{k=1}^N \frac{\mu(2^kB_1)}{\nu_a(2^kB_1)}\leq 4
\end{equation}
(in fact the sum of both terms is at most $4$).
The first estimate yields the property (a) in the lemma with the choice $\nu=\nu_a$. The second estimate tells us that
$$\#\bigg\{k\in [1,N]:\dfrac{\mu(2^kB_1)}{\nu_a(2^kB_1)}\geq 8\eta^{-1}\bigg\}\leq \frac\eta{8}\,\sum_{k=1}^N \frac{\mu(2^kB_1)}{\nu_a(2^kB_1)}
\leq \frac{\eta N}2.$$
Hence, at most there are at least $\eta N - \frac12\eta N = \frac12\eta N$ balls from the family $\{2^jB_1\}_{0\leq j\leq N}$ which are
 $(\mu,C_2)$-good  and such that
$\frac{\mu(2^kB_1)}{\nu_a(2^kB_1)}\leq 8\eta^{-1}$. For these balls $B_k=2^kB_1$ it holds that
\begin{align*}
\int_{B_k} M_n(\chi_{B_k} \nu_a)\,d\nu_a & \leq \|a\|_\infty^2
\int_{B_k} M_n(\chi_{B_k} \mu)\,d\mu \\
& \leq 16C_2 \,\Theta_\mu(B_k)\mu(B_k) \leq 64\eta^{-2}\cdot16C_2 \,\Theta_{\nu_a}(B_k)\nu_a(B_k),
\end{align*}
which gives (b).

To prove (c), for any ball $B$ centered in $R_0\cap\supp\nu_a$ and any $t\in (0,1)$, we consider the function
$$a_t = a(1-t\chi_B)$$
and we write $\nu_t=\nu_{a_t}$. Notice that $\|a_t\|_\infty\leq \|a\|_\infty$. Then we have
$$F(a_t) \leq
\int_{R_0} |\RR\nu_t - c_{R_0}|^p\,d\nu_t + 
\lambda\,\sigma_p(B_1)\bigg(\|a\|_\infty^p + \frac{\mu(B_1)}{\nu_t(B_1)} +  \frac1N\sum_{k=1}^N \frac{\mu(2^kB_1)}{\nu_t(2^kB_1)}\bigg) =: g(t).$$
Notice that $g$ is differentiable in $[0,1)$ and $g(t)\geq F(a_t) \geq F(a) = g(0)$, by the minimizing property of $a$. Hence
$g'(0+)\geq0$. Computing, we get
\begin{align*}
g'(0+) & = 
-\int_{R_0\cap B} |\RR\nu_a - c_{R_0}|^p\,d\nu_a - p \int_{R_0} |\RR\nu_a - c_{R_0}|^{p-2}(\RR\nu_a - c_{R_0})\,\RR(\chi_{R_0\cap B}\nu_a)\,d\nu_a\\
&\quad +
\lambda\,\sigma_p(B_1)\bigg( \frac{\mu(B_1)}{\nu_a(B_1)^2}\,\nu_a(R_0\cap B_1\cap B) +  \frac1N\sum_{k=1}^N \frac{\mu(2^kB_1)}{\nu_a(2^kB_1)^2}
\,\nu_a(R_0\cap 2^kB_1\cap B)\bigg).
\end{align*}
Here we used the fact $\nu_t(2^kB_1) = \mu(2^kB_1\setminus R_0) + \int_{2^kB_1\cap R_0}a(1-t)\chi_B\,d\mu$, and thus
$\frac d{dt}\nu_t(2^kB_1) = -\nu(2^kB_1\cap R_0\cap B)$. The condition $g'(0+)\geq0$ implies that
\begin{align} \label{eqvar5}
\int_{R_0\cap B} &|\RR\nu_a - c_{R_0}|^p\,d\nu_a + p \int_{R_0} |\RR\nu_a - c_{R_0}|^{p-2}(\RR\nu_a - c_{R_0})\,\RR(\chi_{R_0\cap B}\nu_a)\,d\nu_A\\
&\leq
\lambda\,\sigma_p(B_1)\bigg( \frac{\mu(B_1)}{\nu_a(B_1)^2}\,\nu_a(R_0\cap B_1\cap B) +  \frac1N\sum_{k=1}^N \frac{\mu(2^kB_1)}{\nu_a(2^kB_1)^2}
\,\nu_a(R_0\cap 2^kB_1\cap B)\bigg).\nonumber
\end{align}
Notice that the second integral on the right hand side can be written in the form
$$\int_{R_0\cap B} \RR^*_{\nu_a}\big(\chi_{R_0}|\RR\nu_a - c_{R_0}|^{p-2}(\RR\nu_a - c_{R_0})\big)\,d\nu_a.$$
Hence, dividing \rf{eqvar5} by $\nu_a(B)$ and letting $\rad(B)\to0$, using the continuity of 
$|\RR\nu_a - c_{R_0}|^p$ and $\RR^*_{\nu_a}\big(\chi_{R_0}|\RR\nu_a - c_{R_0}|^{p-2}(\RR\nu_a - c_{R_0})\big)$ in the interior of $2B_0$, we deduce that
for all $x\in R_0\cap\supp\nu_a$ which does not belong to the closure of $\supp\nu_a\setminus R_0$, and so for all $x\in\frac54B_0$: 
\begin{align*}
|\RR\nu_a(x) - c_{R_0}|^p +\RR^*_{\nu_a}\big(\chi_{R_0}&|\RR\nu_a - c_{R_0}|^{p-2}(\RR\nu_a - c_{R_0})\big)(x) \\
& \leq 
\lambda\,\sigma_p(B_1)\bigg( \frac{\mu(B_1)}{\nu_a(B_1)^2} +  \frac1N\sum_{k=1}^N \frac{\mu(2^kB_1)}{\nu_a(2^kB_1)^2}
\bigg)\\
& \leq \lambda\,\frac{\sigma_p(B_1)}{\nu_a(B_1)}\bigg( \frac{\mu(B_1)}{\nu_a(B_1)} +  \frac1N\sum_{k=1}^N \frac{\mu(2^kB_1)}{\nu_a(2^kB_1)}\bigg)
\leq 4\lambda\,\frac{\sigma_p(B_1)}{\nu_a(B_1)},
\end{align*}
where, for the last inequality, we used the fact that the sum of both terms in \rf {eqvar3} is at most $4$.
It only remains to notice that
$$\frac{\sigma_p(B_1)}{\nu_a(B_1)} = \Theta_\mu(B_1)^p\,\frac{\mu(B_1)}{\nu_a(B_1)}  \leq 4\,\Theta_\mu(B_1)^p.$$
\end{proof}

\vv

We will need the following lemma.

\begin{lemma}\label{lemCDC}
Let $\Omega\subset  \R^{n+1}$  be a bounded  open set, let $\xi\in \partial\Omega$, and let $r_0>0$ and $r_1=8^{-N}r_0$, for some $N\ge1$. 
 Let $u$ be a nonnegative function which is continuous in
$\overline{B(\xi,r_0)\cap \Omega}$ and subharmonic in $B(\xi,r_0)\cap \Omega$, and vanishes on $B(\xi,r_0)\cap\partial\Omega$.
Suppose that there are $c_3>0$ and a subset $J\subset\{1,\ldots,N\}$ with $\# J\geq \kappa N$, for some $\kappa>0$, such that
\begin{equation}\label{eqvari84}
\capp(\bar B(\xi,8^{-j}r_0)\setminus \Omega) \geq c_3\,(8^{-j}r_0)^{n-1}\quad \mbox{ for all $j\in J$.}
\end{equation}
Then there exists some $\kappa_1>0$ depending on $\kappa$ such that
\begin{equation}
u(x)\lesssim  \bigg(\frac{r_1}{r_0}\bigg)^{\kappa_1}\sup_{\bar B(\xi,r_0)\cap \Omega} u \quad\mbox{ for all  $x\in \Omega\cap B(\xi,r_1)$}
\label{e:holder}
 \end{equation}
\end{lemma}

\vv
The proof of the preceding lemma requires the next well known result.

\begin{lemma}
\label{lembourgain}
 Let $\Omega\subset\R^{n+1}$ be 
a bounded open set. For $x\in\Omega$, denote by $\omega_\Omega^x$ the harmonic measure for $\Omega$ with pole at $x$.
 Let $\bar B$ be a closed ball  centered at $\partial\Omega$. In the case $n\geq2$, we have 
\[ \omega_\Omega^{x}(\bar B)\geq c(d) \,\frac{\capp(\tfrac14 \bar B\setminus\Omega)}{\rad(\bar B)^{d-2}}\quad \mbox{  for all }x\in \tfrac14 \bar B\cap \Omega ,\]
with $c(d)>0$.
In the case $n=1$,
$$ \omega_\Omega^{x}(\bar B)\gtrsim 
 \frac1{\log\dfrac{\capp_L(\bar B)}{\capp_L(\frac14\bar B\setminus\Omega)}} = \frac1{\log\dfrac{\rad(\bar B)}{\capp_L(\frac14\bar B\setminus\Omega)}}\qquad \mbox{ for all $x\in \tfrac14 \bar B\cap \Omega$.}
 $$
\end{lemma}

For the proof see for example Chapter 7 from \cite{Prats-Tolsa}. We remark that in this reference the result is stated for Wiener regular open
sets but essentially the same proof is valid for general open sets.

\vv
\begin{proof}[Proof of Lemma \ref{lemCDC}]
For very $j\geq0$,
let $B_j=B(\xi,8^{-j}r_0)$ and $\Omega_j=\Omega\cap B_j$.
Since $u$ vanishes identically on $\partial\Omega\cap B_j$ and it is continuous in $\overline\Omega_j$ and subharmonic in $\Omega_j$, for all $x\in\partial B_{j+1}\cap\Omega$ we have
$$u(x) \leq \int_{\partial \Omega_j} u(y)\,d\omega_{\Omega_j}^x(y) =
\int_{\partial B_j \cap\Omega} u(y)\,d\omega_{\Omega_j}^x(y) \leq  \omega_{\Omega_j}^x(\partial B_j \cap\Omega)
\, \sup_{\partial B_j \cap\Omega}
u.$$
By Lemma \ref{lembourgain} and the assumption \rf{eqvari84}, in the case $j\in J$ it holds
$$\omega_{\Omega_j}^x(\partial B_j \cap\Omega) = 1- \omega_{\Omega_j}^x(\partial \Omega \cap \bar B_j)
\leq 1-c_0$$
for some $c_0\in (0,1)$ depending on $c_3$. 
 Thus,
$$\sup_{\partial B_{j+1} \cap\Omega}
u \leq (1-c_0)
\, \sup_{\partial B_j \cap\Omega}
u \quad\mbox{ if $j\in J$.}$$
In the case $j\not\in J$ we only have the trivial estimate 
$\omega_{\Omega_j}^x(\partial B_j \cap\Omega) \leq 1$ and
$$\sup_{\partial B_{j+1} \cap\Omega}
u \leq 
\, \sup_{\partial B_j \cap\Omega}
u \quad\mbox{ if $j\not\in J$.}$$

By the maximum principle and iterating, we deduce that
$$\sup_{B_{N} \cap\Omega}
u = \sup_{\partial B_{N} \cap\Omega}
u \leq (1-c_0)^{\#J}
\, \sup_{\partial B_0 \cap\Omega}
u\leq (1-c_0)^{\kappa N}
\, \sup_{\partial B_0 \cap\Omega}
u.$$
This readily proves the lemma, with $\kappa_1=c\kappa$.
\end{proof}

\vv

\begin{lemma}\label{lemmaxpr}
Let $\mu$ and $B_0$ satisfy the assumptions of Main Lemma \ref{lemaclau} and let $R_0$,  
$c_{R_0}$ be as in Lemma \ref{lemvar}.
Suppose that, for some $p\in (1,1+\frac1n)$ and some $\lambda\in (0,1)$,
\begin{equation}\label{eqvar10}
\int_{R_0} |\RR\mu - c_{R_0}|^p\,d\mu \leq \lambda \Theta_\mu(B_1)^p\,\mu(B_1)
\end{equation}
and that, for some $\tau>0$, 
\begin{equation}\label{eqvar11}
\int _{2B_0\setminus R_0}\int_{R_0} \frac1{|x-y|^n}\,d\mu(x)d\mu(y)\leq \tau
\Theta_\mu(B_0)\mu(2B_0).
\end{equation}
For $\nu$ as in Lemma \ref{lemvar}, we have
$$|\RR\nu(x) - c_{R_0}|^p + p\,\RR_\nu^*(\chi_{R_0}|\RR\nu - c_{R_0}|^{p-2}(\RR\nu - c_{R_0}))(x) \leq \big[16\,\lambda 
+ C(\alpha,\eta,\tau)\,\delta_1^{\kappa_1}\big]\,
\Theta_\mu(B_1)^p
$$
for all $x\in 2B_1$, for some $\kappa_1>0$ depending on $C_2$ and $\eta$.
\end{lemma}

\begin{proof}
We consider the function
$$f(x) = \max\big(0,\,|\RR\nu(x) - c_{R_0}|^p + p\,\RR_\nu^*(\chi_{R_0}|\RR\nu - c_{R_0}|^{p-2}(\RR\nu - c_{R_0}))(x)- 16\lambda\,\Theta_\mu(B_1)^p\big).$$
Notice that $f$ is continuous in $\frac32B_0$ and it vanishes in $\supp\nu \cap\frac32B_0$ by Lemma \ref{lemvar}(c). Together with the fact that $f$ is subharmonic $\frac32B_0\setminus \supp\nu$, this implies that
$f$ is subharmonic in the whole $\frac32B_0$

We will show first that $f(x)\lesssim_\alpha \Theta_\mu(B_1)^p$ for all $x\in 1.25 B_0$. 
By the non-negativity and subharmonicity of $f$,  it suffices to check that
\begin{equation}\label{eqmax3}
\avint_{\frac32B_0}f\,d\LL^{n+1}\lesssim_{\alpha,\tau} \Theta_\mu(B_1)^p.
\end{equation}
Let $V=\{x\in  \frac32B_0: f(x)>0\}$, so that $V\subset \frac32B_0\setminus \supp\nu$ and $f=f\chi_V$. For $x\in V$ we have 
\begin{align*}
f(x) & \leq 2 \,|\RR(\chi_{R_0}\nu)(x)|^p  + 
2\,|\RR(\chi_{\R^{n+1}\setminus R_0}\nu)(x) - c_{R_0}|^p \\
& \quad+ p\,\RR_\nu^*(\chi_{R_0}|\RR\nu - c_{R_0}|^{p-2}(\RR\nu - c_{R_0}))(x).
\end{align*}
By H\"older's inequality, 
$$|\RR(\chi_{R_0}\nu)(x)|^p \leq \int_{R_0} \frac1{|x-y|^{np}}\,d\nu(y)\,\,\nu(R_0)^{p-1}.
$$
Then, by Fubini, since $np<n+1$,
\begin{align}\label{eqvar14}
\avint_{\frac32B_0}\chi_V\,|\RR(\chi_{R_0}\nu)|^p\,d\LL^{n+1} & \lesssim \frac{\nu(R_0)^{p-1}}{\rad(B_0)^{n+1}}\int_{\frac32B_0}\int_{R_0} \frac1{|x-y|^{np}}\,d\nu(y)\,dx\\
&
\lesssim \frac{\nu(R_0)^{p-1}\,\nu(R_0)}{\rad(B_0)^{np}} \lesssim \frac{\mu(R_0)^{p}}{\rad(B_0)^{np}} \lesssim \Theta_\mu(B_0)^p.\nonumber
\end{align}

Next we deal with the term $|\RR(\chi_{\R^{n+1}\setminus R_0}\nu)(x) - c_{R_0}|$. To this end, notice that
by the definitions of $\nu$ and $c_{R_0}$,
$$c_{R_0} = m_{\mu,R_0}(\RR\mu) = m_{\mu,R_0}(\RR(\chi_{\R^{n+1}\setminus R_0}\mu)) = m_{\mu,R_0}(\RR(\chi_{\R^{n+1}\setminus R_0}\nu)).$$
Then we have
\begin{align*}
|\RR(\chi_{\R^{n+1}\setminus R_0}\nu)(x) - c_{R_0}| & \leq |\RR(\chi_{\R^{n+1}\setminus 2B_0}\nu)(x) - m_{\mu,R_0}(\RR(\chi_{\R^{n+1}\setminus 2B_0}\nu))|\\
& \quad + |\RR(\chi_{2B_0\setminus R_0}\nu)(x) - m_{\mu,R_0}(\RR(\chi_{2B_0\setminus R_0}\nu))|.
\end{align*}
By standard arguments, the first term on the right hand side does not exceed $C\PP_\mu(B_0)\leq C\,C_0\Theta_\mu(B_0)$.
Taking also into account the assumption \rf{eqvar11},
\begin{align*}
|m_{\mu,R_0}(\RR(\chi_{B_0\setminus R_0}\nu))| & \leq \frac4{\mu(R_0)} \int _{2B_0\setminus R_0}\int_{R_0} \frac1{|x-y|^n}\,d\mu(x)d\mu(y)\\
& \leq 4\tau\,\Theta_\mu(B_0)\frac{\mu(2B_0)}{\mu(R_0)}\lesssim\tau\,\Theta_\mu(B_0).
\end{align*}
Thus,
$$|\RR(\chi_{\R^{n+1}\setminus R_0}\nu)(x) - c_{R_0}|  \lesssim  C_0\Theta_\mu(B_0) + \tau\,\Theta_\mu(B_0) + |\RR(\chi_{2B_0\setminus R_0}\nu)(x)|.$$
Therefore,
$$\avint_{\frac32B_0}\!\chi_V\,|\RR(\chi_{\R^{n+1}\setminus R_0}\nu) - c_{R_0}|^p\,d\LL^{n+1}
\lesssim (C_0+\tau)^p\, \Theta_\mu(B_0)^p + \avint_{\frac32B_0} \!|\RR(\chi_{2B_0\setminus R_0}\nu)|^p\,d\LL^{n+1}.$$
Arguing as in \rf{eqvar14}, it follows easily that the last term is also bounded above by
$C\Theta_\mu(B_0)^p$. Hence, we get
$$\avint_{\frac32B_0}|\RR(\chi_{\R^{n+1}\setminus R_0}\nu) - c_{R_0}|^p\,d\LL^{n+1}
\lesssim_\tau  \Theta_\mu(B_0)^p.$$

Finally we have to estimate
$$\avint_{\frac32B_0} \chi_V\,\RR_\nu^*(\chi_{R_0}|\RR\nu - c_{R_0}|^{p-2}(\RR\nu - c_{R_0}))|^p\,d\LL^{n+1}.$$
By Fubini, this equals
$$\frac1{\LL^{n+1}(\frac32B_0)} \int_{R_0} \RR_{\LL^{n+1}}(\chi_V) \,|\RR\nu - c_{R_0}|^{p-2}(\RR\nu - c_{R_0}))|^p\,d\nu.$$
Using that $|\RR_{\LL^{n+1}}(\chi_V)|\lesssim \rad(B_0)$, we derive 
\begin{align*}
\bigg|\avint_{\frac32B_0} \chi_V\,\RR_\nu^*(\chi_{R_0}|\RR\nu - c_{R_0}|^{p-2}(\RR&\nu - c_{R_0}))|^p\,d\LL^{n+1}\bigg| 
\lesssim \frac1{\rad(B_0)^n} \int_{R_0} |\RR\nu - c_{R_0}|^{p-1}\,d\nu\\
& \quad\lesssim \frac{\nu(R_0)^{1/p}}{\rad(B_0)^n} \bigg(\int_{R_0} |\RR\nu - c_{R_0}|^{p}\,d\nu\bigg)^{1/p'} \\
& \quad\lesssim \frac{\lambda^{1/p'}
\nu(R_0)^{1/p}\Theta_\mu(B_1)^{p-1}\mu(B_1)^{1/p'}}{\rad(B_0)^n}\\
& \quad\lesssim \lambda^{1/p'}\Theta_\mu(B_0)^{1/p}\Theta_\mu(B_1)^{p-1}\Theta_\mu(B_1)^{1/p'} \lesssim_\alpha \Theta_\mu(B_1).
\end{align*}
Gathering the above estimates and taking into account that $\Theta_\mu(B_0)\leq \alpha^{-1}\Theta_\mu(B_1)$, we get \rf{eqmax3}. So $f(x)\lesssim_{\alpha,\tau} \Theta_\mu(B_1)^p$ for all $x\in 1.25 B_0$.

Let $\wt B_1$ be a maximal ball concentric with $B_1$ contained in $1.25B_0$. Since $B_1$ is centered in $B_0$, we have $\rad(\wt B_1)\geq \frac14\rad(B_0)$. By Lemma \ref{lemvar}, at least $\eta\,N/2$ balls from the family $\{2^jB_1\}_{0\leq j\leq N}$ are $(\nu,1024\eta^{-1}C_2)$-good, for
 $N= \lfloor\log_2 (\rad(B_0)/\rad(B_1)\rfloor$. 
Applying Lemma \ref{lemCDC} to the function $f$ in $\wt B_1$, with $N$ large enough if necessary, we deduce that, for all $x\in 2B_1$,
$$f(x)\lesssim \bigg(\frac{\rad(2B_1)}{\rad(\wt B_1)}\bigg)^{\kappa_1} \sup_{\wt B_1}f \lesssim_{\alpha,\eta,\tau} \bigg(\frac{\rad(B_1)}{\rad(B_0)}\bigg)^{\kappa_1}\Theta_\mu(B_1)^p \lesssim_{\alpha,\eta,\tau} \delta_1^{\kappa_1}\,\Theta_\mu(B_1)^p,$$
for some $\kappa_1>0$ depending on $C_2$ and $\eta$. By the definition of $f$, this concludes the proof of the lemma.
\end{proof}
\vv

\begin{proof}[Proof of Main Lemma \ref{lemaclau}]
We will prove this by contradiction. Let $\mu$, $B_0$ and $R_0$, $c_{R_0}$ as in Lemma \ref{lemvar}.
Suppose that 
\begin{equation}\label{eqvar20}
\int_{R_0} |\RR\mu - m_{\mu,R_0}(\RR\mu)|^2\,d\mu \leq \lambda_0 \Theta_\mu(B_0)^2\,\mu(B_0)
\end{equation}
for some small $0<\lambda_0<1$. Then this implies that, for $1<p<2$,
\begin{align*}
\int_{R_0} |\RR\mu - c_{R_0}|^p\,d\mu & \leq \left(\int_{R_0} |\RR\mu - m_{\mu,R_0}(\RR\mu)|^2\,d\mu\right)^{\frac p2}\mu(R_0)^{1-\frac p2}\\
& \leq 
\left(\int_{R_0} |\RR\mu - m_{\mu,2B_0}(\RR\mu)|^2\,d\mu\right)^{\frac p2}\mu(R_0)^{1-\frac p2} \leq
C\lambda_0^{p/2} \,\Theta_\mu(B_0)^p\,\mu(B_0)\\
&\leq 
C(\alpha,\delta_0)\lambda_0^{p/2} \,\Theta_\mu(B_1)^p\,\mu(B_1).
\end{align*}
Set $\lambda= C(\alpha,\delta_0)\lambda_0^{p/2}$ and take $p\in (1,1+\frac1n)$. Consider the measure $\nu$ in Lemma \ref{lemvar}, so that it 
satisfies the properties (a), (b), (c) in that lemma, and also
\begin{equation}\label{eqvar22}
|\RR\nu(x) - c_{R_0}|^p + p\,\RR_\nu^*(\chi_{R_0}|\RR\nu - c_{R_0}|^{p-2}(\RR\nu - c_{R_0}))(x) \leq \gamma\,
\Theta_\mu(B_1)^p
\end{equation}
for all $x\in 2B_1$, for 
\begin{equation}\label{eqvargamma}
\gamma= 16\lambda +  C(\alpha,\eta,\tau)\,\delta_1^{\kappa_1}
\end{equation}
 for some ${\kappa_1}>0$ depending on $C_2$ and $\eta$, by Lemma \ref{lemmaxpr}.

Consider a function $\vphi_1\in C^\infty$ supported in $2B_1$ which is identically $1$ on $B_1$, with $\|\nabla\vphi_1\|_\infty\lesssim 1/\rad(B_1)$, and take 
$$\vphi = \Theta_\mu(B_1)\,\vphi_1.$$
Then we have
$$\RR^*_{\LL^{n+1}}(\nabla \vphi) = c\,\vphi,$$
for some absolute constant $c$, and so
\begin{align}\label{eqvar25}
\left|\int (\RR\nu - c_{R_0})\,\nabla\vphi \,d\LL^{n+1}\right|  & =  \left|\int \RR\nu \,\nabla\vphi \,d\LL^{n+1}\right| =
\left|\int  \RR_{\LL^{n+1}}^*(\nabla\vphi) \,d\nu\right|\\& \approx \int  \vphi \,d\nu \geq \Theta_\mu(B_1)\,\nu(B_1) \geq \frac14\,\Theta_\mu(B_1)\,\mu(B_1).\nonumber
\end{align}
On the other hand, taking into account that 
$$\|\nabla\vphi\|_1\lesssim \Theta_\mu(B_1)\,\rad(B_1)^n = \mu(B_1),$$
we infer that
\begin{align}\label{eqvar27}
\left|\int (\RR\nu - c_{R_0})\,\nabla\vphi \,d\LL^{n+1}\right|  & \leq
\bigg(\int |\RR\nu - c_{R_0}|^p\,|\nabla\vphi| \,d\LL^{n+1}\bigg)^{1/p} \bigg(\int |\nabla\vphi| \,d\LL^{n+1}\bigg)^{1/p'} \\
&\lesssim \bigg(\int |\RR\nu - c_{R_0}|^p\,|\nabla\vphi| \,d\LL^{n+1}\bigg)^{1/p}\mu(B_1)^{1/p'}.\nonumber
\end{align}
We estimate the integral on right hand side using \rf{eqvar22}:
\begin{align*}
&\int |\RR\nu - c_{R_0}|^p\,|\nabla\vphi| \,d\LL^{n+1} \\
&\leq C \gamma\Theta_\mu(B_1)^p
\int |\nabla\vphi| \,d\LL^{n+1} - p\int\RR_\nu^*(\chi_{R_0}|\RR\nu - c_{R_0}|^{p-2}(\RR\nu - c_{R_0}))\,|\nabla\vphi| \,d\LL^{n+1}\\
& \leq C
\gamma\Theta_\mu(B_1)^p\,\mu(B_1) - p\int_{R_0}|\RR\nu - c_{R_0}|^{p-2}(\RR\nu - c_{R_0}))\,\RR_{\LL^{n+1}}(|\nabla\vphi|)\,d\nu.
\end{align*}
To bound the last integral on the right hand side we use the fact that, for any $x\in\R^{n+1}$, 
$$|\RR_{\LL^{n+1}}(|\nabla\vphi|)(x)| \lesssim \frac{\Theta_\mu(B_1)}{\rad(B_1)}\int_{2B_1}\frac1{|x-y|^n}\,dy\lesssim \Theta_\mu(B_1).$$
Then we have
\begin{align*}
\bigg|\int_{R_0}|\RR\nu - c_{R_0}&|^{p-2}(\RR\nu - c_{R_0}))\,\RR_{\LL^{n+1}}(|\nabla\vphi|)\,d\nu\bigg|\\
& \lesssim \Theta_\mu(B_1)
\int_{R_0}|\RR\nu - c_{R_0}|^{p-1}\,d\nu\\
& \leq \Theta_\mu(B_1)\left(\int_{R_0}|\RR\nu - c_{R_0}|^{p}\,d\nu\right)^{1/p'} \nu(R_0)^{1/p}\\
& \leq \lambda^{1/p'}\Theta_\mu(B_1)\,\Theta_\mu(B_1)^{p-1}\,\mu(B_1)^{1/p'}\nu(R_0)^{1/p}
\lesssim_{\alpha,\delta_0} \lambda^{1/p'}\Theta_\mu(B_1)^p\,\mu(B_1).
\end{align*}
Therefore,
$$
\int |\RR\nu - c_{R_0}|^p\,|\nabla\vphi| \,d\LL^{n+1} \leq 
C\,\gamma\,\Theta_\mu(B_1)^p\,\mu(B_1) + C(\alpha,\delta_0)\lambda^{1/p'}\,\Theta_\mu(B_1)^p\,\mu(B_1).$$
By \rf{eqvar27}, this gives
$$
\left|\int (\RR\nu - c_{R_0})\,\nabla\vphi \,d\LL^{n+1}\right|  
\lesssim 
(\gamma + C(\alpha,\delta_0)\lambda^{1/p'})^{1/p}\,\Theta_\mu(B_1)\,\mu(B_1).
$$
Together with \rf{eqvar25}, this gives
$$\Theta_\mu(B_1)\,\mu(B_1)\lesssim (\gamma + C(\alpha,\delta_0)\lambda^{1/p'})^{1/p}\,\Theta_\mu(B_1)\,\mu(B_1).$$
That is,
$$1\lesssim (\gamma + C(\alpha,\delta_0)\lambda^{1/p'})^{1/p} = \big(16\lambda + C(\alpha,\delta_0)\lambda^{1/p'}+ C(\alpha,\eta,\tau)\,\delta_1^{\kappa_1}\big)^{1/p},$$
which is a contradiction if $\delta_1$ and $\lambda$ are taken small enough.\footnote{Remark that $\delta_1$ depends on $\tau$.}
\end{proof}
\vv


\section{The proof of Theorem \ref{teomain}}

We will first prove the following ``toy version'' of Theorem \ref{teomain}.

\begin{propo}\label{mainpropo}
Let $\mu$ be a Radon measure in $\R^{n+1}$ and let $B_0$ be a $(\PP_\mu,C_0)$-doubling ball. Suppose that $M_n(\chi_{3B_0}\mu)\in L^2(\mu|_{2B_0})$ and 
$\RR_*(\chi_{3B_0}\mu)\in L^2(\mu|_{2B_0})$. Suppose that there exists some ball $B_1$ centered in $B_0$ such that
$$\Theta_\mu(B_1)\geq \alpha\,\Theta_\mu(B_0)$$
for some $\alpha>0$, and such that, for some $\delta_0,\delta_1>0$,
$$\delta_0\,\rad(B_0)\leq \rad(B_1)\leq \delta_1\,\rad(B_0),$$
with $\delta_1\in (0,1/2)$ small enough, depending on $n$, $C_0$, and $\alpha$.
Let $N= \lfloor\log_2 (\rad(B_0)/\rad(B_1)\rfloor$. Suppose that at least $\eta N$ balls from the family $\{2^jB_1\}_{0\leq j\leq N}$ are 
$(\mu,C_2)$-good, for some $\eta>0$.
Then, if $\delta_1$ is  small enough, depending on $n$, $C_0$, $\alpha$, $C_2$, and $\eta$, we have
\begin{equation}\label{eqpropo11}
\int_{2B_0} |\RR\mu - m_{\mu,2B_0}(\RR\mu)|^2\,d\mu + \int_{2B_0} \theta_\mu^{n,*}(x)^2\,d\mu(x)\geq c_2 \Theta_\mu(B_0)^2\,\mu(B_0),
\end{equation}
for some constant $c_2>0$ depending on 
$n$, $\alpha$, $\delta_0$, and $C_2$.
\end{propo}

\begin{proof}
We may assume that
\begin{equation}\label{eqasu7}
\int_{2B_0} |\RR\mu - m_{\mu,2B_0}(\RR\mu)|^2\,d\mu + \int_{2B_0} \theta_\mu^{n,*}(x)^2\,d\mu(x)\leq \Theta_\mu(B_0)^2\,\mu(B_0)
\end{equation}
because otherwise we are done.

For $k$ large enough, consider the measure $\wt\sigma_k$ constructed in \rf{eqaprox99}.
It is easy to check that if some ball $2^jB_1$, with $0\leq j\leq N$ is $(\mu,C_2)$-good, then it will be $(\wt \sigma_k,C_2/2)$-good
for $k$ large enough. Further, since the family of the cubes $Q\in I_k$ is finite, by construction it is clear that there exists some $h_k\in L^\infty(\mu)$
such that $\wt \sigma_k = h_k\,\LL^{n+1}$ in $2B_0$. Obviously, this implies that $\theta_{\wt\sigma_k}^{n,*}(x)=0$ $\wt\sigma_k$-a.e.\ in $2B_0$.
Also, 
 $M_n(\chi_{3B_0}\wt\sigma_k)\in L^2(\wt\sigma_k|_{2B_0})$ and $\RR_*(\chi_{3B_0}\wt\sigma_k)\in L^2(\wt\sigma_k|_{2B_0})$, by Lemma \ref{lemmapprox}.
By H\"older's inequality and by Remark \ref{rem2}, we have
\begin{align*}
\int _{2B_0\setminus R_0}\int_{R_0} &\frac1{|x-y|^n}\,d\wt\sigma_k(x)d\wt\sigma_k(y) \\
&\leq
\bigg(\int _{2B_0\setminus R_0}\bigg(\int_{R_0} \frac1{|x-y|^n}\,d\wt\sigma_k(x)\bigg)^2 d\wt\sigma_k(y)\bigg)^{1/2}\wt\sigma_k(2B_0)^{1/2}\\
& \lesssim
\bigg(\Theta_{\wt \sigma_k} (B_0)^2\wt\sigma_k(2B_0) + \int_{2B_0} \!|\RR\wt\sigma_k - m_{\wt\sigma_k,2B_0}(\RR\wt\sigma_k)|^2 \,d\wt\sigma_k\bigg)^{1/2}\wt\sigma_k(2B_0)^{1/2}.
\end{align*}
For $k$ large enough, by Lemma \ref{lemmapprox} and the assumption \rf{eqasu7}.
\begin{align*}
\Theta_{\wt \sigma_k} (B_0)^2&\,\wt\sigma_k(2B_0) + \int_{2B_0} \!|\RR\wt\sigma_k - m_{\wt\sigma_k,2B_0}(\RR\wt\sigma_k)|^2 \,d\wt\sigma_k\\
&
\lesssim 
\Theta_{\mu} (B_0)^2\mu(B_0) + \int_{2B_0} \!|\RR\mu - m_{\mu,2B_0}(\RR\mu)|^2 \,d\mu+\int_{2B_0} \theta_\mu^{n,*}(x)^2\,d\mu(x)\\& \lesssim \Theta_{\mu} (B_0)^2\mu(B_0).
\end{align*}
Therefore,
$$\int _{2B_0\setminus R_0}\int_{R_0} \frac1{|x-y|^n}\,d\wt\sigma_k(x)d\wt\sigma_k(y)\lesssim \Theta_{\mu} (B_0)\mu(B_0)\approx \Theta_{\wt\sigma_k} (B_0)^2\wt\sigma_k(B_0).$$
So the assumptions of Main Lemma \ref{lemaclau} hold for $\wt\sigma_k$. The application of this lemma tells us that
\begin{align*}
\int_{2B_0} |\RR\wt\sigma_k - m_{\wt\sigma_k,2B_0}(\RR\wt\sigma_k)|^2\,d\wt\sigma_k &\geq
\int_{R_0} |\RR\wt\sigma_k - m_{\wt\sigma_k,R_0}(\RR\wt\sigma_k)|^2\,d\wt\sigma_k\\
&\geq c_1 \Theta_{\wt\sigma_k}(B_0)^2\,\wt\sigma_k(B_0)\approx\Theta_{\mu}(B_0)^2\,\mu(B_0)
.
\end{align*}
A new application of Lemma \ref{lemmapprox} yields \rf{eqpropo11}.
\end{proof}

\vv

We will need the following auxiliary result.

\begin{lemma}\label{lemdobpp}
Let $\mu$ be a Radon measure and let $B\subset \R^{n+1}$ be a ball. 
Suppose that $B$ is not $(\PP_\mu,4)$-doubling. Then
$$\PP_\mu(2B) > \frac32\,\PP_\mu(B).$$
\end{lemma}

\begin{proof}
	The condition $\PP_\mu(2B) > \frac32\,\PP_\mu(B)$ is equivalent to the inequality
$$\sum_{j\geq1} 2^{-j+1}\Theta_\mu(2^jB) > \frac32 \sum_{j\geq0} 2^{-j}\Theta_\mu(2^jB) = \frac32\Theta_\mu(B) + 
\frac34 \sum_{j\geq1} 2^{-j+1}\Theta_\mu(2^jB).$$
	That is, 
	$$\frac16 \sum_{j\geq1} 2^{-j+1}\Theta_\mu(2^jB) > \Theta_\mu(B),$$
	or
	\begin{equation}\label{eqpmu1}
	\frac13\sum_{j\geq1} 2^{-j}\Theta_\mu(2^jB) > \Theta_\mu(B),
	\end{equation}

	In turn, the fact that $B$ is not $(\PP_\mu,4)$-doubling is equivalent to
	$$\Theta_\mu(B)<\frac14 \sum_{j\geq0} 2^{-j}\Theta_\mu(2^jB) = \frac14\Theta_\mu(B) + \frac14\sum_{j\geq1} 2^{-j}\Theta_\mu(2^jB),$$
	which gives \rf{eqpmu1}.
\end{proof}

\vv

\begin{lemma}\label{lembonescales}
Let $\mu$ be a Radon measure in $\R^{n+1}$ and let $B_0$ be a $(\PP_\mu,C_0)$-doubling ball. Suppose that there exists some ball $B_1$ centered in $B_0$ such that
$$\Theta_\mu(B_1)\geq \alpha\,\Theta_\mu(B_0)$$
for some $\alpha>0$, and such that, for some $\delta_0,\delta_1>0$,
$$\delta_0\,\rad(B_0)\leq \rad(B_1)\leq \delta_1\,\rad(B_0),$$
with $\delta_1\in (0,1/100)$.
Let $N= \lfloor\log_2 (\rad(B_0)/\rad(B_1)\rfloor$. If $N$ is large enough, depending on $C_0$ and $\alpha$, then
at least $\eta N$ balls from the family $\{2^jB_1\}_{1\leq j\leq N}$ are $(\PP_\mu, 4)$-doubling, with 
$$\eta = \frac1{1+3n}.$$ 
\end{lemma}

\begin{proof}
Let $I_\PP\subset \{1,\ldots,N\}$ be the subset of the indices $j$ such that $2^jB_1$ is $(\PP_\mu, 4)$-doubling. By the preceding lemma, if
$j\not\in I_\PP$, then 
$$\PP_\mu(2^jB_1) \leq \frac23\,\PP_\mu(2^{j+1}B_1).$$
On the other hand, if $j\in I_\PP$, then
$$\PP_\mu(2^jB_1) = \sum_{k\geq0} 2^{-k}\Theta_\mu(2^{j+k}B) \leq 2^n\sum_{k\geq0} 2^{-k}\Theta_\mu(2^{j+1+k}B) = 2^n\,\PP_\mu(2^{j+1}B_1).$$ 
Suppose that there are $\theta N$ balls $B_j$, with $j=1,\ldots,N$, which are not $(\PP_\mu, 4)$-doubling.
Then,
\begin{align*}
\alpha\,\Theta_\mu(B_0)& \leq \PP_\mu(B_1) \lesssim \PP_\mu(2B_1) \\& \leq \bigg(\frac23\bigg)^{\theta N} \big(2^n\big)^{(1-\theta)N}\,\PP_\mu(2^{N+1}B_1)
\lesssim \bigg(\frac23\bigg)^{\theta N} \big(2^n\big)^{(1-\theta)N}\,C_0\,\Theta_\mu(B_0).
\end{align*}
Thus,
$$\bigg(\frac32\bigg)^{-\theta N} \bigg(\frac32\bigg)^{2n(1-\theta)N}\geq
\bigg(\frac32\bigg)^{-\theta N} \big(2^n\big)^{(1-\theta)N}\geq  c\alpha C_0^{-1} ,$$
which implies
$$-\theta N +2n(1-\theta)N \geq  \frac{\log(c\alpha C_0^{-1})}{\log{3/2}}=: -C_3(\alpha,C_0),$$
or equivalently,
$$\theta\leq  \frac{2n}{1+2n} + \frac{C_3(\alpha,C_0)}{(1+2n)N}.$$
So, for $N$ large enough depending on $ C_3(\alpha,C_0)$,
$$1-\theta \geq \frac{1}{1+2n} - \frac{C_3(\alpha,C_0)}{(1+2n)N} \geq \frac{1}{1+3n},$$
which proves the lemma.
\end{proof}

\vv

\begin{proof}[Proof of Theorem \ref{teomain}] 
We may assume that
$$\int_{2B_0} |\RR\mu - m_{\mu,2B_0}(\RR\mu)|^2\,d\mu + \int_{2B_0} \theta_\mu^{n,*}(x)^2\,d\mu(x)\leq \lambda \Theta_\mu(B_0)^2\,\mu(B_0),
$$
for some small $\lambda>0$ to be fixed below, which may depend on $\delta_0$ and $\alpha$.
By the previous lemma, 
 there at least $\eta N$ balls from the family $\{B_j\}_{j\in J} \subset \{2^jB_1\}_{0\leq j\leq N}$ that are $(\PP_\mu,4)$-doubling
 with $\eta=\frac1{1+3n}$.
We claim that the same balls $B_j$, with $j\in J$, are also $(\mu,C_2)$-good, for some constant $C_2$ depending only on $n$. To prove this,
we apply Theorem \ref{teoDTlocal00} to such balls $B_j$, and then we deduce that 
\begin{align*}
\int_{B_j} |M_n(\mu|_{B_j})|^2\, d\mu 
 & \lesssim 
\int_{2B_j} |\RR\mu - m_{\mu,2B_j}(\RR\mu)|^2\,d\mu + \PP_\mu(B_j)^2\mu(2B_j) + \int_{2B_j} \theta_\mu^{n,*}(x)^2\,d\mu(x)\\
& \leq
\int_{2B_0} |\RR\mu - m_{\mu,2B_0}(\RR\mu)|^2\,d\mu + 16\,\Theta_\mu(B_j)^2\mu(2B_j) + \int_{2B_0} \theta_\mu^{n,*}(x)^2\,d\mu(x)\\
& \leq \lambda \,\Theta_\mu(B_0)^2\mu(2B_0) + 16\,\Theta_\mu(B_j)^2\mu(2B_j).
\end{align*}
Remark that the implicit constant in the first inequality above only depends on $n$.
For $\lambda$ small enough depending on $\alpha$ and $\delta_0$, using the fact that $\mu(B_j)\geq \mu(B_1)$, we have 
$$\lambda \,\Theta_\mu(B_0)^2\mu(2B_0) \leq \Theta_\mu(B_j)^2\mu(B_j).$$
Also, the fact that $B_j$ is $(\PP_\mu,4)$-doubling implies that $\Theta_\mu(2B_j)\leq 8\,\Theta_\mu(B_j)$ and so
$\mu(2B_j)\leq 2^n\,8\,\mu(B_j)$. Hence,
$$\int_{B_j} |M_n(\mu|_{B_j})|^2\, d\mu \leq C_5\Theta_\mu(B_j)^2\mu(B_j),$$
with $C_5$ depending only on $n$.
This implies that
$$\int_{B_j} M_n(\mu|_{B_j})\, d\mu\leq \bigg(\int_{B_j} |M_n(\mu|_{B_j})|^2\, d\mu\bigg)^{1/2}\mu(B_j)^{1/2} \leq C_5^{1/2}\Theta_\mu(B_j)\,\mu(B_j),$$
proving our claim.

Now we only have to notice that the assumptions of Proposition \ref{mainpropo} hold, and so
$$\int_{2B_0} |\RR\mu - m_{\mu,2B_0}(\RR\mu)|^2\,d\mu + \int_{2B_0} \theta_\mu^{n,*}(x)^2\,d\mu(x)\geq c_2 \Theta_\mu(B_0)^2\,\mu(B_0),$$
with $c_2>0$ depending on 
$n$, $\alpha$, $\delta_0$, and $C_2$.
\end{proof}
\vv

\begin{rem}\label{remtonto}
By modifying suitable the parameters in the proof of Theorem \ref{teomain}, one deduces that, under the assumptions of the theorem, for any
$\lambda\in (1,2)$, 
one can deduce that there exists some constant $c_1>0$ depending on 
$n$, $\alpha$, $\delta_0$, and $\lambda$ such that
$$c_1 \Theta_\mu(B_0)^2\,\mu(B_0)\leq \int_{\lambda B_0} |\RR\mu - m_{\mu,2B_0}(\RR\mu)|^2\,d\mu + \int_{\lambda B_0} \theta_\mu^{n,*}(x)^2\,d\mu(x).$$
\end{rem}
\vv


\section{Proof of Corollary \ref{cororeflec}}

Let $\mu$ be a Radon measure in $\R^{n+1}$ with polynomial $n$-growth which is reflectionless and denote
$$E= \{x\in\supp(\mu):\theta_\mu^{n,*}(x)>0\big\}.$$
Notice that $\mu$ is mutually absolutely continuous with the Hausdorff measure $\HH^n$ on $E$. Then, the fact that $E$ is $n$-rectifiable follows from \cite{NToV2}.

We will show that if $E$ is not dense in $\supp(\mu)$, then $\dim_H(\mu)>n+\ve_n$, for some absolute constant $\ve_n>0$. To this
end, consider an open set $U$ such that $\mu(U)>0$ and $E\cap U = \varnothing$. It suffices to show that there exists a subset $F\subset U$
with $\mu(F)>0$ such that, for some positive constants $C$, $r_0$, $\ve_n'$,
$$\mu(B(x,r))\leq C\,r^{n+\ve_n'}\quad\mbox{ for all $x\in F$, $0<r\leq r_0$.}$$

Applying Theorem \ref{teomain} with 
 $C_0=4$, $\alpha=1/{2^{n+3}}$, and $\delta_0=\delta_1^{2n+2}$ we deduce the following. There exists some absolute constant $\delta_1\in(0,1/2)$ such that if $B_0$ is a $(\PP_\mu,4)$-doubling ball centered in 
$\supp(\mu)$ and contained in $U$, and moreover
 there exists  some ball $B_1$ centered in $B_0$ such that
$$\Theta_\mu(B_1)\geq \frac1{2^{n+3}}\,\Theta_\mu(B_0)$$
with 
$$\delta_1^{2n+2}\,\rad(B_0)\leq \rad(B_1)\leq \delta_1\,\rad(B_0),$$
then 
$$c_1 \Theta_\mu(B_0)^2\,\mu(B_0)\leq \int_{2B_0} |\RR\mu - m_{\mu,2B_0}(\RR\mu)|^2\,d\mu,$$
for some absolute constant $c_1>0$.
So $\mu$ cannot be a reflectionless measure. Without loss of generality we can assume that $\delta_1=2^{-m}$ for some $m\ge3$.

We call a ball $B_0$ good, if it satisfies the conditions in the previous paragraph. That is,
$B_0$ is a $(\PP_\mu,4)$-doubling ball centered in 
$\supp(\mu)$ and contained in $U$, and moreover
 there exists  some ball $B_1$ centered in $B_0$ such that
$$\Theta_\mu(B_1)\geq \frac1{2^{n+3}}\,\Theta_\mu(B_0)$$
with 
$$\delta_1^{2n+2}\,\rad(B_0)\leq \rad(B_1)\leq \delta_1\,\rad(B_0).$$
From the preceding discussion, such balls do not exist. 

Therefore, for any ball $B$ contained in $U$ an centered in $U$ one of the following
options holds:
\begin{itemize}
\item[(i)] $B$ is not $(\PP_\mu,4)$-doubling, or
\item[(ii)] $B$ is $(\PP_\mu,4)$-doubling but $B$ is not good, which implies that
 all the balls $B_1$ concentric with $B$ such that
$$\delta_1^{2n+2}\,\rad(B)\leq \rad(B_1)\leq \delta_1\,\rad(B),$$
satisfy $\Theta_\mu(B_1)< \frac1{2^{n+3}}\,\Theta_\mu(B)$.
\end{itemize}
In the first case, by Lemma \ref{lemdobpp}, 
$$\PP_\mu(B)\leq \frac23\,\PP_\mu(2B)\leq \frac34\,\PP_\mu(2B) .$$
On the other hand, in the case (ii), all the balls of the form
$2^{-k}B$ with $m\leq k\leq (2n+2)m$ satisfy $$\Theta_\mu(2^{-k}B)\leq \frac1{2^{n+3}}\,\Theta_\mu(B)\leq \frac18\,\Theta_\mu(2B).$$
We claim that this implies that
\begin{equation}\label{eqpmu6d1}
\PP_\mu(2^{-(2n+2)m}B)\leq \frac34\,\PP_\mu(2B).
\end{equation}

To prove the claim, we denote $\wt B= 2^{-(2n+2)m}B$ and we split
\begin{align*}
\PP_\mu(\wt B) & = \sum_{k=0}^{(2n+1)m-1}  2^{-k}\Theta_\mu(2^k\wt B) + \sum_{k=(2n+1)m}^{(2n+2)m}  2^{-k}\Theta_\mu(2^k\wt B) + 
\sum_{k>(2n+2)m}  2^{-k}\Theta_\mu(2^k\wt B)\\
& =: I + II + III.
\end{align*}
The balls $2^k\wt B$ in the sum $I$ satisfy 
$$\delta_1^{2n+2}B = \wt B\subset 2^k\wt B \subset 2^{(2n+1)m-1}\wt B = \frac{\delta_1}2\,B,$$
and so
$$\sum_{k=0}^{(2n+1)m-1}  2^{-k}\Theta_\mu(2^k\wt B)\leq \frac1{2^{n+3}}\sum_{k\geq 0}  2^{-k}\,\Theta_\mu(B)  = \frac1{2^{n+2}}\,\Theta_\mu(B)\leq \frac14\,\Theta_\mu(2B).$$
To estimate the term $II$ we use the fact that the balls $2^k\wt B$, with $(2n+1)m\leq k\leq(2n+2)m$ are contained in $B$ and have
radius at least $\delta_1\rad(B)$, and so
$$\Theta_\mu(2^k\wt B) \leq \frac{\rad(B)^n}{(\delta_1\rad(B))^n}\Theta_\mu(B)= \delta_1^{-n}\Theta_\mu(B).$$
Then, using also that $m\geq 3$, we derive
$$II\leq \sum_{k\geq(2n+1)m}2^{-k}\delta_1^{-n}\Theta_\mu(B)\leq 2^{-(2n+1)m+1}\delta_1^{-n}\Theta_\mu(B) = 2^{-mn-m+1}\Theta_\mu(B)
\leq \frac14\,\Theta_\mu(2B).$$
Finally, to estimate the term $III$ we use the fact that, for $k>(2n+2)m$, the balls  $2^k\wt B$ contain $2B$:
\begin{align*}
III&=\sum_{k>(2n+2)m}  2^{-k}\Theta_\mu(2^k\wt B) = \sum_{k\geq1}  2^{-k-(2n+2)m}\Theta_\mu(2^kB)\\
& \leq  2^{-(2n+2)m+1}\PP_\mu(2B)
\leq \frac14\,\PP_\mu(2B).
\end{align*}
Gathering the preceding estimates, we get
$$\PP_\mu(\wt B)\leq \frac14\Theta_\mu(2B) + \frac14\Theta_\mu(2B) + \frac14\,\PP_\mu(2B) \leq \frac34\,\PP_\mu(2B),$$
which proves \rf{eqpmu6d1}.
By combining the options (i) and (ii), we deduce that either
\begin{itemize}
\item[(i)] $\PP_\mu(B)\leq \frac34\,\PP_\mu(2B)$, or 
\item[(ii)] $\PP_\mu(2^{-(2n+2)m}B)\leq \frac34\,\PP_\mu(2B)$.
\end{itemize}

Consider now a ball $\widehat B$ centered in $\supp(\mu)$ such that $2\wh B\subset U$. For any $x\in \wh B\cap\supp(\mu)$, let $B_1(x)$ be a ball 
centered in $x$ with radius equal to $\rad(\wh B)$. Now we define inductively $B_k(x)$, for $k>1$, as follows: 
in case that the option (i) holds for $\frac12B_k(x)$ in place of $B$, we set $B_{k+1} = \frac12\,B_k(x)$.
On the other hand, if (ii) holds for $\frac12B_k(x)$ in place of $B$, we set $B_{k+1}(x)= 2^{-(2n+2)m-1}B_k(x)$.
In this way, in any case we have
$$\PP_\mu(B_{k+1}(x))\leq \frac34\,\PP_\mu(B_{k}(x)),$$
and
\begin{equation}\label{eqmn6fy}
2^{-(2n+2)m-1}\rad(B_{k}(x))\leq \rad(B_{k+1}(x))\leq \frac12 \rad(B_{k}(x)).
\end{equation}
This implies that, for any $k\geq1$,
\begin{equation}\label{eqdk83a}
\Theta_\mu(B_k(x))\leq \PP_\mu(B_{k}(x))\leq \Big(\frac34\Big)^{k-1}\,\PP_\mu(B_{1}(x))\lesssim \Big(\frac34\Big)^{k}\,\PP_\mu(\wh B)
\end{equation}
and
$$2^{(-k+1)[(2n+2)m-1]}\rad(B_1(x))\leq \rad(B_k(x))\leq 2^{-k+1}\rad(\wh B(x)).$$
We write
$$2^{(-k+1)[(2n+2)m-1]} = C \bigg(\frac34\bigg)^{k\beta},$$
for some suitable $C,\beta>0$, so that $ \rad(B_k(x))\geq C \left(\frac34\right)^{k\beta}\rad(\wh B)$.
 From the last estimate and \rf{eqdk83a} we deduce that
$$\mu(B_k(x)) \lesssim \bigg(\frac34\bigg)^{k}\,\rad(B_k(x))^n \,\PP_\mu(\wh B)\lesssim \rad(B_k(x))^{n+\frac1\beta}\,\rad(\wh B)^{-\frac1\beta}\,\PP_\mu(\wh B),$$

For an arbitrary radius $r\in (0,\rad(\wh B))$, let $k\geq 1$ be such that 
$\rad(B_{k+1}(x))< r\leq \rad(B_{k}(x))$. By \rf{eqmn6fy}, 
$r\approx \rad(B_{k}(x))$, with the implicit constant depending on $m$ and $n$. So we have
$$\mu(B(x,r))\leq \mu(B_k(x))\lesssim \rad(B_k(x))^{n+\frac1\beta}\,\rad(\wh B)^{-\frac1\beta} \,\PP_\mu(\wh B)\approx r^{n+\frac1\beta}\,\rad(\wh B)^{-\frac1\beta}\,\PP_\mu(\wh B)$$
for all $x\in \wh B\cap\supp(\mu)$. Since $\mu(\wh B\cap\supp(\mu))>0$, we infer that
$$\dim_H(\mu)\geq n+\frac1\beta,$$
which concludes the proof of the corollary.
\qed

\vv


\section{The proof of Theorem \ref{teomain2}}

\subsection{A partial result}

Concerning Theorem \ref{teomain2}, first we will show that the existence of an $n$-rectifiable (not necessarily uniformly rectifiable) set 
 $\Gamma$ such that $\mu(\Gamma\cap B_0) \geq\tau\,\mu(B_0)$ is an easy consequence of Corollary \ref{corotxulo} and some results on quantitative
 rectifiability. Indeed, let $\gamma\in(0,1/10)$ be some small constant to be fixed below. Suppose first that
 the set
 $$F= \big\{x\in 2B_0:\theta_\mu^{n,*}(x)>0\big\}$$
satisfies
$$\mu(F)\leq \gamma\,\mu(B_0).$$
Together with the assumption (c) in the theorem, by Corollary \ref{corotxulo}, this implies that
\begin{align}\label{eqfpv79}
\int_{B_0}\int_0^{2\rad(B_0)} \beta_{2,\mu}^n(x,r)^2&\,\Theta_\mu(x,r)\,\frac{dr}r\,d\mu(x)\\ &
 \lesssim
\int_{2B_0} |\RR\mu - m_{\mu,2B_0}(\RR\mu)|^2\,d\mu + \int_{2B_0\setminus F} \theta_\mu^{n,*}(x)^2\,d\mu(x)\nonumber\\
& \lesssim
(\ve + \gamma)\,\Theta_\mu(B_0)^2\,\mu(B_0),\nonumber
\end{align}
Now, if we assume both $\ve$ and $\gamma$ small enough, the existence of the aforementioned rectifiable set $\Gamma$ is a direct
consequence of the following result:

\begin{theorem}\label{teoAT}
Let $\mu$ be  a Radon measure in $\R^d$ supported in a ball $B_0$ such that
$$\Theta_\mu(x,r) \leq C\,\Theta_\mu(B_0)\quad\mbox{ for all $x\in\R^d$, $r>0$.}$$
Suppose that
$$\int_{B_0}\!\int_0^{2\rad(B_0)} \beta_{\mu,2}^n(x,r)^2\,\Theta_\mu(x,r)\,\frac{dr}r\,d\mu(x) \leq \ve\,\Theta_\mu(B_0)^2\,\mu(B_0).$$
If $\ve$ is small enough, then there exists a uniformly $n$-rectifiable set $\Gamma\subset B_0$ such that
$$\mu(B_0 \cap \Gamma)\geq \frac1{10}\,\mu(B_0).$$
\end{theorem}

This result is from \cite{Azzam-Tolsa}, although it is not stated explicitly in this work. Instead, this follows easily from Main Lemma 4.1 and Lemma 7.2 from \cite{Azzam-Tolsa}.
 
In the case when $\mu(F)> \gamma\,\mu(B_0)$, we apply the following result, also from \cite{Azzam-Tolsa}:

\begin{theorem}\label{teo1*}
Let $p\geq 0$ and let $\mu$ be a Radon measure in $\R^d$ 
such that $0<\theta^{n,*}(x,\mu)<\infty$ for $\mu$-a.e.\ $x\in\R^d$. If
\begin{equation}\label{eqjones**}
\int_0^1 \beta_{\mu,2}(x,r)^2\,\Theta_\mu(x,r)^p\,\frac{dr}r<\infty \quad\mbox{ for $\mu$-a.e.\ $x\in\R^d$,}
\end{equation}
then $\mu$ is $n$-rectifiable.
\end{theorem}

By applying this theorem to $\mu|_F$, we deduce that $F$ is $n$-rectifiable, and then we can take $F=\Gamma$.

\vv

\subsection{The low density cubes and the approximating measure $\eta$}

The arguments for the proof of the full Theorem \ref{teomain2}, which requires the set $\Gamma$ to be uniformly $n$-rectifiable, are
more lengthy. We will follow the same scheme as in \cite{Girela-Tolsa}: we use the David-Mattila lattice and then we show that a suitable family of low density cubes has small measure. Afterwards, we approximate the measure $\mu$ by another $n$-Ahlfors regular measure $\wt\mu$ such that
the Riesz transform $\RR_{\wt\mu}$ is bounded in $L^2(\wt\mu)$, which implies that $\wt\mu$ is uniformly rectifiable, by \cite{NToV}.

We assume that we are under the assumptions of Theorem \ref{teomain2}.
We consider the David-Mattila lattice associated with $\mu$ and we denote by $\FF_{\max} $ the family of maximal cubes from $\DD_\mu$ that 
intersect $1.5 B_0$ and are contained in $1.8B_0$, so that the cubes $Q\in \FF_{\max} $ satisfy $\ell(Q)\approx\rad(B_0)$.
We denote 
\begin{equation}\label{eqR0910}
R_0 = \bigcup_{Q\in \FF_{\max}} Q,
\end{equation}
so that
$$\frac32 B_0\cap \supp(\mu)\subset R_0\subset 2B_0.$$

Let $0<\theta_0\ll1$ be a very small constant to be fixed later. We denote by $\LD$ the family of those cubes from 
$\DD_\mu$ which are contained in $R_0$ such that $\Theta_\mu(3.5B_Q)\leq \theta_0\,\Theta_\mu(B_0)$ and have maximal side length.

We will prove the following (compare with Key Lemma 6.1 from \cite{Girela-Tolsa}):

\begin{keylemma}\label{keylemma}
Assume $\theta_0\in(0,1/10)$ small enough.
Then there exists some constant $\ve_0$ such that if 
 $\delta_1\in (0,1/2)$ is small enough, depending on $n$, $C_0$, $C_1$, and $\alpha$, and
$\ve$ is small enough, depending on $C_0$, $C_1$, $\alpha$, $\theta_0$, and $\delta_0$, then
\begin{equation}\label{eqkeylemma}
\mu\Biggl(\bigcup_{Q\in \LD} Q\Biggr)\leq (1-\ve_0)\,\mu(R_0).
\end{equation}
\end{keylemma}

We will deduce this result from Theorem \ref{teomain} and Remark \ref{remtonto} with $\lambda=\frac43$, applied to a suitable approximating measure $\eta$.
To this end, we will argue by contradiction, assuming that 
\begin{equation}\label{eqkeylemmacontra}
\mu\Biggl(\bigcup_{Q\in \LD} Q\Biggr)> (1-\ve_0)\,\mu(R_0)
\end{equation}
Next we need to construct another family of stopping cubes which we will denote by $\sss$. This is
defined as follows. For each $Q\in\LD$ we consider the family of maximal cubes contained in $Q$ 
from $\DD_\mu^{db}$ 
with side length at most $\theta_0^{1/(n+1)}\ell(Q)$. We denote this family by 
$\sss(Q)$.
Then we define
$$\sss = \bigcup_{Q\in\LD} \sss(Q).$$
 Note that, by the properties of the lattice $\DD_\mu$, for each $Q\in\LD$, the cubes from $\sss(Q)$ cover $\mu$-almost all $Q$.
 So  \rf{eqkeylemmacontra} is equivalent to 
$$\mu\Biggl(\bigcup_{Q\in \sss} Q\Biggr)> (1-\ve_0)\,\mu(R_0)
$$
The family $\sss$ may consist of an infinite number of cubes.
For technical reasons, it is convenient to consider a finite subfamily of $\sss$ which contains a very big proportion of the $\mu$ measure of union of the cubes from $\sss$. So we let $\sss_0$ be a {\em finite} subfamily of $\sss$
such that
\begin{equation}\label{stop00}
\mu\Biggl(\bigcup_{Q\in \sss_0} Q\Biggr)> (1- 2\ve_0)\,\mu(R_0).
\end{equation}
Recall also that, by construction the cubes  $P\in\sss_0$ satisfy $\ell(P)\lesssim \theta_0^{1/(n+1)}\rad(B_0)$.

\begin{lemma}\label{lempoisson}
We have:
$$\Theta_\mu(2B_Q)\leq \PP_\mu(Q)\lesssim\theta_0^{1/(n+1)}\Theta_\mu(B_0)\quad \mbox{ for all $Q\in\sss$.}$$
\end{lemma}

The proof of this lemma is essentially the same as the one of Lemma 6.2 from \cite{Girela-Tolsa} and we skip it.

Now we will define an auxiliary measure $\mu_0$. First, given a small constant $0<\kappa_0\ll1$ (to be fixed below) and $Q\in\DD_\mu$, we consider
the following ``inner region'' of $Q$: 
\begin{equation}\label{eqik00}
I_{\kappa_0}(Q) = \{x\in Q:\dist(x,\supp\mu\setminus Q)\geq \kappa_0\ell(Q)\}.
\end{equation}
 We set
\begin{equation}\label{eqdefmu0}
\mu_0 = \mu|_{\R^{n+1}\setminus R_0} + \sum_{Q\in\sss_0} \mu|_{I_{\kappa_0}(Q)}.
\end{equation}
By the doubling property and small boundary condition \rf{eqfk490} of $Q\in\sss_0$, we have
$$\mu(Q\setminus I_{\kappa_0}(Q))\lesssim \kappa_0^{1/2}\,\mu(3.5B_Q) \lesssim \kappa_0^{1/2}\,\mu(Q).$$
Then we get the following estimate for the total variation of 
$\mu-\mu_0$:
\begin{align}\label{eqfac99}
\|\mu-\mu_0\| & =
\mu(R_0) - \mu_0(R_0)\\ & = \bigg(\mu(R_0) - \sum_{Q\in \sss_0}\mu(Q)\bigg)+ \sum_{Q\in \sss_0}\big(\mu(Q) - \mu(I_{\kappa_0}(Q))\big)\nonumber\\
&\lesssim 2\ve_0 \,\mu(R_0) + \kappa_0^{1/2}
\sum_{Q\in \sss_0}\mu(Q) \leq (2\ve_0 + \kappa_0^{1/2})\,\mu(R_0).\nonumber
\end{align}
We deduce the following (compare with Lemma 6.4 from \cite{Girela-Tolsa}):

\begin{lemma}\label{lemfac93}
We have
$$
\int_{2B_0} |\RR\mu_0 - m_{\mu_0,2B_0}(\RR\mu_0)|^2\,d\mu_0\lesssim (\ve+\ve_0+\kappa_0^{1/2})\,\Theta_\mu(B_0)^2\,\mu(B_0).$$
\end{lemma}

\begin{proof}
By \rf{eqfac99} and the fact that $\|\RR_\mu\|_{L^2(\mu|_{2B_0})\to L^2(\mu|_{2B_0})}\leq C_1\Theta_\mu(B_0)$, we have
\begin{align*}
\int_{2B_0}|\RR&\mu_0 - m_{\mu_0,2B_0}(\RR\mu_0)|^2\,d\mu_0 \\
& \leq
\int_{2B_0} |\RR\mu_0 - m_{\mu,2B_0}(\RR\mu)|^2\,d\mu_0 \\
& \leq 2\int_{2B_0} |\RR\mu - m_{\mu,2B_0}(\RR\mu)|^2\,d\mu_0 + 2\int_{2B_0} |\RR\mu_0 - \RR\mu|^2\,d\mu_0\\
& \leq 2\ve\,\Theta_\mu(B_0)^2\,\mu(B_0) + C_1\Theta_\mu(B_0)^2\,\|\mu-\mu_0\|\lesssim (\ve+\ve_0+ C_1\kappa_0^{1/2})\,\Theta_\mu(B_0)^2\,\mu(B_0).
\end{align*}
\end{proof}

\vv

\begin{lemma}\label{lemaux23}
For all $Q\in\sss$, we have
$$ \int_{1.1B_Q\setminus Q}\int_Q
\frac1{|x-y|^n}\,d\mu_0(x)\,d\mu_0(y)\lesssim 
\theta_0^{\frac1{2(n+1)^2}}\Theta_\mu(B_0)
\,\mu_0(Q). $$
\end{lemma}

The proof is essentially the same as the one of \cite[Lemma 7.3]{Girela-Tolsa}.

We consider now another auxiliary measure:
$$\eta = \mu_0|_{\R^{n+1}\setminus R_0} + \sum_{Q\in\sss_0} \mu_0(Q) \,\frac{\HH^{n+1}|_{\tfrac14B  (Q)}}{\HH^{n+1}\bigl(\tfrac14B  (Q)\bigr)}.$$
In a sense, $\eta$ can be considered as an approximation of $\mu_0$ which is absolutely continuous with respect to $\HH^{n+1}$ on the low density cubes. 
By Remark 5.2 from \cite{Girela-Tolsa}, the balls $\frac12B(Q)$, $Q\in\sss$, are pairwise disjoint.
So the balls $\tfrac14B  (Q)$ in the sum above satisfy
$$\dist(\tfrac14B  (Q),\tfrac14B  (Q'))\geq \frac14\,\bigl[\rad(B(Q))+ \rad(B(Q'))\bigr]\quad
\mbox{ if $Q\neq Q'$, for $Q,Q'\in\sss$}.$$
We denote 
\begin{equation}\label{eqS0def}
S_0 = \bigcup_{Q\in\sss_0} \tfrac14B  (Q),
\end{equation}
so that 
$$\supp(\eta)\subset \overline{\R^{n+1}\setminus R_0} \cup {S_0},$$
with the union being disjoint. Here we assumed the balls $\frac12 B(Q)$, $Q\in\sss_0$, to be closed.

The following result should be compared to Lemma 8.1 from \cite{Girela-Tolsa}.

\begin{lemma}\label{lem200}
Under the assumption \rf{eqkeylemmacontra}, the following hold:
$$\int_{S_0}|\RR\eta - m_{\eta,S_0}(\RR\eta)|^2\,d\eta\lesssim \ve'\,\Theta_\mu(B_0)^2\eta(S_0),$$
where
$\ve'= \ve +\ve_0+ \kappa_0^{1/2} + \kappa_0^{-2n-2}\,\theta_0^{\frac1{(n+1)^2}}$.
\end{lemma}

Notice that  $\RR_\eta$ is a bounded operator in $L^2(\eta|_{2B_0})$. Indeed, on $S_0$, $\eta$ is absolutely continuous with respect to the Lebesgue measure and has a bounded density, since the number of
cubes in $\sss_0$ is bounded. This implies that $\RR_\eta$ is bounded in $L^2(\eta|_{S_0})$ (although we do not have a good quantitative control on its norm). Also the fact that $\RR_{\mu_0}$ is bounded in $L^2(\mu_0|_{2B_0})$ implies that $\RR_\eta$ is bounded in $L^2(\eta|_{2B_0\setminus S_0})$. Consequently, $\RR_\eta$ is bounded in $L^2(\eta|_{2B_0})$ (see Proposition 2.25 from \cite{Tolsa-llibre}, for example).  
On the other hand, the fact that $\RR_\eta$ is a bounded operator $L^2(\eta|_{2B_0})$ implies that $\RR\eta$ is well defined $\eta$-a.e.\ in $S_0$ as a principal value in a BMO sense and that
\begin{equation}\label{eqeta739}
\RR_*(\chi_{2B_0}\eta)\in L^2(\eta|_{2B_0}).
\end{equation}

Although the proof of Lemma \ref{lem200} is quite similar to the one of \cite[Lemma 8.1]{Girela-Tolsa}, for completeness we will show the detailed arguments.

\begin{proof}[Proof of Lemma \ref{lem200}]
We consider the function
$$f = \chi_{\R^{n+1}\setminus R_0} \RR\mu_0+ \sum_{Q\in \sss_0} m_{\mu_0,Q}(\RR\mu_0)\,\chi_Q.$$
Recall that the cubes from $\sss_0$ are contained in $R_0\subset 2B_0$. It is clear that 
$$c_{R_0}:= m_{\mu_0,R_0}(\RR\mu_0) = m_{\mu_0,R_0}(f),$$
and also, by Lemma \ref{lemfac93},
\begin{equation}\label{eqsk592}
\int_{R_0}|f - c_{R_0}|^2\,d\mu_0\leq \int_{R_0}|\RR\mu_0 - m_{\mu_0,R_0}(\RR\mu_0)|^2\,d\mu_0
\lesssim (\ve+\ve_0+\kappa_0^{1/2})\,\Theta_\mu(B_0)^2\,\mu(B_0).
 \end{equation}
 
Next we will estimate $|\RR\eta(x) - c_{R_0}|$ for $x\in\supp(\eta)\cap R_0$. Recall that $S_0$ is defined in \rf{eqS0def}. For $x\in S_0$ and $Q\in\sss_0$ such that $x\in\tfrac14B  (Q)$, we write
\begin{align}\label{eqt123*}
\bigl|\RR\eta(x) - c_{R_0}\bigr| & \leq \bigl|\RR(\chi_{\tfrac14B  (Q)}\eta)(x)\bigr| 
+  \bigl|\RR(\chi_{\R^{n+1}\setminus \tfrac14B  (Q)}\eta)(x) - \RR(\chi_{\R^{n+1}\setminus Q}\mu_0)(x)\bigr|\\
&\quad + \bigl|\RR(\chi_{\R^{n+1}\setminus Q}\mu_0)(x)- m_{\mu_0,Q}(\RR\mu_0)\bigr|
+\bigl|m_{\mu_0,Q}(\RR\mu_0)- c_{R_0}\bigr|\nonumber\\
& =: T_1 + T_2 + T_3 + \bigl|m_{\mu_0,Q}(\RR\mu_0)- c_{R_0}\bigr|.\nonumber
\end{align}
Using that $$\eta|_{\tfrac14B  (Q)} = \mu_0(Q)\,\frac{\HH^{n+1}|_{\tfrac14B  (Q)}}{\HH^{n+1}\bigl(\tfrac14B  (Q)\bigr)},$$
together with Lemma \ref{lempoisson}, it follows easily that
$$T_1=\bigl|\RR(\chi_{\tfrac14B  (Q)}\eta)(x)\bigr|\lesssim \frac{\mu_0(Q)}{\rad(B(Q))^n}\lesssim \theta_0^{1/(n+1)}\Theta_\mu(B_0).$$

Next we will deal with the term $T_3$ in \rf{eqt123*}.
To this end, for $x\in\tfrac14B  (Q)$  we set
\begin{align}\label{eqsd539}
\bigl|\RR(&\chi_{\R^{n+1}\setminus Q}\mu_0)(x) - m_{\mu_0,Q}(\RR\mu_0)\bigr|\\
& \leq \bigl|\RR(\chi_{1.1B_Q\setminus Q}\mu_0)(x)\bigr| +
\bigl|\RR(\chi_{\R^{n+1}\setminus 1.1B_Q}\mu_0)(x) - m_{\mu_0,Q}(\RR(\chi_{\R^{n+1}\setminus 1.1B_Q}\mu_0))\bigr| \nonumber\\
&\quad+ \bigl|m_{\mu_0,Q}(\RR(\chi_{1.1B_Q\setminus Q}\mu_0))\bigr|,\nonumber
\end{align}
taking into account that $m_{\mu_0,Q}(\RR(\chi_Q\mu_0))=0$, by the antisymmetry of the Riesz kernel. 
The first term on the right hand side satisfies
$$\bigl|\RR(\chi_{1.1B_Q\setminus Q}\mu_0)(x)\bigr|\leq \int_{1.1B_Q\setminus Q}\frac1{|x-y|^n}\,d\mu_0(y)
\lesssim \frac{\mu_0(1.1B_Q)}{\rad(B(Q))^n}\lesssim \theta_0^{1/(n+1)}\Theta_\mu(B_0),$$
recalling that $x\in \frac14B(Q)$ and that $\Theta_\mu(1.1B_Q)\lesssim \theta_0^{1/(n+1)}\Theta_\mu(B_0)$ for the last estimate, by Lemma \ref{lempoisson}.

Now we turn our attention to the second term on the right hand side of \rf{eqsd539}. For $x'\in Q\in\sss_0$, 
we have
\begin{align*}
\bigl|\RR(\chi_{\R^{n+1}\setminus 1.1B_Q}\mu_0)(x) -\RR&(\chi_{\R^{n+1}\setminus 1.1B_Q}\mu_0)(x')\bigr|\\
&\leq \int_{\R^{n+1}\setminus 1.1B_Q} \bigl|K(x-y)-K(x'-y)\bigr|\,d\mu_0(y)
\lesssim \PP_{\mu_0}(2B_Q),
\end{align*}
where $K(\cdot)$ stands for the kernel of the Riesz transform.
Taking into account that the distance  both from $x$ and $x'$ to $(1.1B_Q)^c$ is larger than $c\,\rad(B_Q)$. 
Averaging on $x'\in Q$ with respect to $\mu_0$ we get
$$\bigl|\RR(\chi_{\R^{n+1}\setminus 1.1B_Q}\mu_0)(x) - m_{\mu_0,Q}(\RR(\chi_{\R^{n+1}\setminus 1.1B_Q}\mu_0))\bigr|
\lesssim \PP_{\mu_0}(2B_Q).$$

To estimate the last term in \rf{eqsd539} we just apply Lemma \ref{lemaux23}:
$$ \bigl|m_{\mu_0,Q}(\RR(\chi_{1.1B_Q\setminus Q}\mu_0))\bigr|\leq \frac1{\mu_0(Q)} \int_{1.1B_Q\setminus Q}\int_Q
\frac1{|x-y|^n}\,d\mu_0(x)\,d\mu_0(y)\lesssim \theta_0^{\frac1{2(n+1)^2}}\,\Theta_\mu(B_0).$$
Thus, we obtain
\begin{align}\label{eq1934}
T_3& =
\bigl|\RR(\chi_{\R^{n+1}\setminus Q}\mu_0)(x) - m_{\mu_0,Q}(\RR\mu_0)\bigr|\lesssim \theta_0^{\frac1{n+1}}\Theta_\mu(B_0) + 
 \PP_{\mu_0}(2B_Q) + \theta_0^{\frac1{2(n+1)^2}}\Theta_\mu(B_0)\\
 & \lesssim \theta_0^{\frac1{2(n+1)^2}}\Theta_\mu(B_0) + \PP_{\mu_0}(2B_Q).\nonumber
\end{align}

To deal with the term $T_2$ in \rf{eqt123*}, for $x\in\tfrac14B  (Q)$ we write
\begin{align}\label{eqt2***}
T_2 & = \bigl|\RR(\chi_{\R^{n+1}\setminus \tfrac14B  (Q)}\eta)(x) - \RR(\chi_{\R^{n+1}\setminus Q}\mu_0)(x)\bigr| \\
& \leq \sum_{R\in \sss_0:R\neq Q} \left|\int K(x-y)\,d(\eta_{\frac14B  (R)} - \mu_0|_R)\right|\nonumber\\
& \leq
\sum_{R\in \sss_0:R\neq Q} \int |K(x-y) - K(x-x_R)|\,d(\eta_{\frac14B  (R)} + \mu_0|_R),\nonumber
\end{align}
where $x_R$ is the center of $R$ and we used that $\eta(\tfrac14B  (R)) = \mu_0(R)$ for the last inequality.

We claim that, for $x\in \tfrac14B  (Q)$ and $y\in\tfrac14B  (R)\cup \supp(\mu_0|_R)$, with $Q$ and $R$ as in \rf{eqt2***},
\begin{equation}\label{eqdqr1}
|K(x-y) - K(x-x_R)|\lesssim \frac{\ell(R)}{\kappa_0^{n+1}\,D(Q,R)^{n+1}},
\end{equation}
where 
$$D(Q,R)=\ell(Q) + \ell(R) + \dist(Q,R).$$
To show \rf{eqdqr1} note first that
\begin{equation}\label{eq19050}
x\in \tfrac14B  (Q),\;\;x_R\in \tfrac14B  (R)\quad\Rightarrow \quad \quad|x-x_R|\gtrsim D(Q,R),
\end{equation}
since $\frac12B(Q)\cap \frac12B(R)=\varnothing$. Analogously, because of the same reason,
\begin{equation}\label{eq1905}
x\in \tfrac14B  (Q),\;\;y\in \tfrac14B  (R)\quad\Rightarrow \quad|x-y|\gtrsim D(Q,R).
\end{equation}
 Also,
\begin{equation}\label{eq1906}
x\in \frac14B(Q),\;\;y\in \supp(\mu_0|_R)\quad\Rightarrow \quad|x-y|\gtrsim \kappa_0 D(Q,R),
\end{equation}
To prove this, note that
\begin{equation}\label{eq1903}
y\in \supp(\mu_0|_R) = I_{\kappa_0}(R)\subset R,
\end{equation}
which implies that $y\not\in B(Q)$ and thus $|x-y|\geq \frac34 \rad(B(Q))\approx\ell(Q)$. In the case
$\rad(B(Q))\geq 2\kappa_0\ell(R)$
 this implies that
$$|x-y|\gtrsim \ell(Q) + \kappa_0\ell(R).$$
Otherwise, when
$\rad(B(Q))< 2\kappa_0\ell(R)$,
from \rf{eq1903}, the fact that $x_Q\in\supp\mu_0$ and $y\in I_{\kappa_0}(R)$, and the definition of $I_{\kappa_0}(R)$, we get
$$|x_Q-y|\geq \kappa_0\ell(R),$$
and then, as $|x_Q-x|\leq \frac14\rad(B(Q))\leq \frac12\kappa_0\ell(R)$, we infer that
$$|x-y|\geq |x_Q-y| - |x_Q-x|\geq \frac{\kappa_0}2\ell(R).$$
So in any case we have $|x-y|\gtrsim \kappa_0(\ell(Q) + \ell(R))$. It is easy to deduce
\rf{eq1906} from this estimate. We leave the details for the reader.

From \rf{eq19050}, \rf{eq1905}, and \rf{eq1906}, and the fact that $K(\cdot)$ is a standard
Calder\'on-Zygmund kernel, we get \rf{eqdqr1}.
Plugging this estimate into \rf{eqt2***}, we obtain
\begin{equation}\label{eqnos5800}
T_2\lesssim \frac1{\kappa_0^{n+1}}
\sum_{R\in \sss_0} \frac{\ell(R)\,\mu_0(R)}{D(Q,R)^{n+1}}.
\end{equation}

So from \rf{eqt123*} and the estimates for the terms $T_1$, $T_2$ and $T_3$, we infer that
for all $x\in\tfrac14B  (Q)$ with $Q\in\sss_0$,
\begin{equation}\label{eqnos58}
\bigl|\RR\eta(x)- c_{R_0}\bigr|  \lesssim  \bigl|m_{\mu_0,Q}(\RR\mu_0)- c_{R_0}\bigr| +
\theta_0^{\frac1{2(n+1)^2}} \,\Theta_\mu(B_0)+ \PP_{\mu_0}(2B_Q)+
\frac1{\kappa_0^{n+1}}\!\!
\sum_{R\in \sss_0}\!\! \frac{\ell(R)\,\mu_0(R)}{D(Q,R)^{n+1}}.
\end{equation}
Denote 
$$ \wt p _{\mu_0}(x) = \!\sum_{Q\in \sss_0}\!\chi_{\frac14B(Q)}\,  \PP_{\mu_0}(2B_Q) \quad\mbox{and}\quad
\wt g(x) =\!
\sum_{Q,R\in \sss_0} \frac{\ell(R)}{D(Q,R)^{n+1}} \,\mu_0(R)\,\chi_{\frac14B(Q)}(x).$$ 

Squaring and integrating \rf{eqnos58} with respect to $\eta$ on $S_0$,
we get
\begin{align}\label{eqff593}
\bigl\|\RR\eta- c_{R_0}\bigr\|_{L^2(\eta|_{S_0})}^2  & \lesssim  
\sum_{Q\in\sss_0} \bigl|m_{\mu_0,Q}(\RR\mu_0)- c_{R_0}\bigr|^2\,\eta(\tfrac14B(Q)) \\ &\quad +
\theta_0^{\frac1{(n+1)^2}}\,\Theta_\mu(B_0)^2\,\eta(S_0) + \|\wt p_{\mu_0}\|_{L^2(\eta|_{S_0})}^2+ \frac1{\kappa_0^{2n+2}}\|\wt g\|_{L^2(\eta|_{S_0})}^2.\nonumber
\end{align}
Note that, since $\eta(\tfrac14B(Q))=\mu_0(Q)$, the first sum on the right hand side of \rf{eqff593} 
equals $\|f- c_{R_0}\|_{L^2(\mu_0|_{R_0})}^2$, which does not exceed $(\ve+\ve_0+\kappa_0)\,\eta(R_0)$, by \rf{eqsk592}.
By an analogous argument we 
deduce that $\|\wt p_{\mu_0}\|_{L^2(\eta|_{S_0})}^2 = \| p_{\mu_0}\|_{L^2(\mu_0|_{R_0})}^2$ 
and $\|\wt g\|_{L^2(\eta|_{R_0})}^2 = \| g\|_{L^2(\mu_0|_{R_0})}^2$, where
$$
 p _{\mu_0}(x) = \sum_{Q\in \sss_0}\chi_Q\,  \PP_{\mu_0}(2B_Q) \quad\mbox{ and }\quad g(x) = \sum_{Q,R\in\sss_0}
\frac{\ell(R)}{D(Q,R)^{n+1}} \,\mu_0(R)\,\chi_Q(x).
$$

We will estimate $\|g\|_{L^2(\mu_0|_{R_0})}$ by duality: for any non-negative function $h\in L^2(\mu_0|_{R_0})$,
we set
\begin{align}\label{eqgh57}
\int g\,h\,d\mu_0 &= 
\sum_{Q\in\sss_0}
\sum_{R\in \sss_0} \frac{\ell(R)}{D(Q,R)^{n+1}} \,\mu_0(R)\,\int_Q h\,d\mu_0\\
& = \sum_{R\in \sss_0} \mu_0(R) \sum_{Q\in\sss_0}
 \frac{\ell(R)}{D(Q,R)^{n+1}} \,\int_Q h\,d\mu_0.\nonumber
 \end{align}
 For each $R\in \sss_0$ and $z\in R$ we have 
\begin{align*}
\sum_{Q\in\sss_0} \frac{\ell(R)}{D(Q,R)^{n+1}} \,\int_Q h\,d\mu_0 &\lesssim
\int \frac{\ell(R)\,h(y)}{\bigl(\ell(R)+|z-y|\bigr)^{n+1}}\,d\mu_0(y)\\
& = \int_{|z-y|\leq \ell(R)}\cdots \,\,\,+ \sum_{j\geq1}\int_{2^{j-1}\ell(R)<|z-y|\leq 2^{j-1}\ell(R)}\cdots \\
& \lesssim 
\sum_{j\geq0}\; \avint_{B(z,2^{j}\ell(R))}h\,d\mu_0 \;\,\frac{2^{-j}\,\mu_0(B(z,2^j\ell(R)))}{\bigl(2^j\ell(R)
\bigr)^n} \\
&\lesssim M_{\mu_0}h(z)\,\PP_{\mu_0}\bigl(B(z,\ell(R))\bigr),
\end{align*}
where $M_{\mu_0}$ stands for the centered maximal Hardy-Littlewood operator with respect to $\mu_0$. Then, by \rf{eqgh57},
\begin{align*}
\int g\,h\,d\mu_0 &\lesssim
 \sum_{R\in \sss_0} \inf_{z\in R} \big[M_{\mu_0}h(z)\,\PP_{\mu_0}\bigl(B(z,\ell(R))\bigr)\bigl]\,
 \mu_0(R) \lesssim \int_{R_0}M_{\mu_0}h\,\,p_{\mu_0} \,d\mu_0
 \\ 
 & \lesssim \|M_{\mu_0}h\|_{L^2(\mu_0)}\,\| p_{\mu_0} \|_{L^2(\mu_0|_{R_0})}
 \lesssim \|h\|_{L^2(\mu_0)}\,\| p_{\mu_0} \|_{L^2(\mu_0|_{R_0})}
 .
 \end{align*}
Thus, defining $p_\mu$ in the same way as $p_{\mu_0}$, with $\mu_0$ replaced by $\mu$, by Lemma \ref{lempoisson},
\begin{equation}\label{eqgg991}
\|\wt g\|_{L^2(\eta|_{R_0})}^2 =
\|g\|_{L^2(\mu_0|_{R_0})}^2\lesssim\|p_{\mu_0} \|_{L^2(\mu_0|_{R_0})}^2 \leq \|p_{\mu} \|_{L^2(\mu|_{R_0})}^2
\lesssim
\theta_0^{2/(n+1)}\Theta_\mu(B_0)^2\mu(R_0).
\end{equation}
Plugging this estimate into \rf{eqff593}, using \rf{eqsk592}, and recalling that $\|\wt g\|_{L^2(\eta|_{R_0})}
=\|g\|_{L^2(\mu_0|_{R_0})}$, we obtain
\begin{align}\label{eqshbm410}
\bigl\|\RR\eta - c_{R_0}\bigr\|_{L^2(\eta|_{S_0})}^2 & \lesssim 
\biggl(\ve+\ve_0+\kappa_0^{1/2} +\theta_0^{\frac1{(n+1)^2}}+ \frac{1}{\kappa_0^{2n+2}}\,\theta_0^{\frac2{n+1}}\biggr)\Theta_\mu(B_0)^2
\,\eta(R_0)\\
& \lesssim \biggl(\ve+\ve_0+\kappa_0^{1/2} + \frac{1}{\kappa_0^{2n+2}}\,\theta_0^{\frac1{(n+1)^2}}\biggr)
\,\Theta_\mu(B_0)^2\,\eta(R_0),\nonumber
\end{align}
\vv
which implies the lemma.
\end{proof}
\vv

\begin{proof}[Proof of the Key Lemma \ref{keylemma}]
Recall that we are assuming \rf{eqkeylemmacontra} and we want to reach a contradiction.
 We intend to apply Theorem \ref{teomain} to the measure $\eta$, with the ball $B_0$ replaced by $1.1B_0$ and $2B_0$ by $\frac32 B_0$
 (using Remark \ref{remtonto}).
To check that the assumptions of Theorem \ref{teomain} hold, remark first that 
 $1.1B_0$ is  $(\PP_\eta,C\,C_0)$, for $\ve_0$ and $\kappa_0$ small enough in 
\rf{eqfac99}.

Recall that $\ell(P)\lesssim \theta_0^{1/(n+1)}\rad(B_0)$ for the cubes 
$P\in\sss_0$.
From the assumptions (a) and (b) in the theorem and the construction of $\eta$,  with $\ve_0,\theta_0,\kappa_0$ small enough, and the fact that 
$$\|\mu-\mu_0\| \leq (2\ve_0 + \kappa_0^{1/2})\,\mu(R_0)\approx (2\ve_0 + \kappa_0^{1/2})\,\eta(1.1B_0),$$
we deduce that
$$\sup_{0<r\leq 4\rad(B_0)} \Theta_\eta(x,r) \lesssim C_1\Theta_\eta(B_0) \quad\mbox{ for all $x\in 2B_0$},$$
and that there exists a ball $B_1'$ centered in $1.1B_0$ 
 such that, for some positive constants $\alpha$, $\delta_0$, $\delta_1$,
$$\Theta_\eta(B_1')\gtrsim\alpha\,\Theta_\eta(1.1B_0),$$
with
$$\delta_0\,\rad(B_0)\leq \rad(B_1')\leq 2\delta_1\,\rad(B_0).$$

By Lemma \ref{lem200}, using the fact that
 $\frac32 B_0\cap\supp(\eta)\subset S_0$ and that $\eta(S_0)\approx\eta(1.1B_0)$, since $\eta(S_0)=\mu(R_0)\approx\mu(B_0)$, we have
\begin{equation}\label{eqetaM9f}
\int_{\frac32B_0} |\RR\eta - m_{\eta,\frac32B_0}(\RR\eta)|^2\,d\eta
\lesssim \int_{S_0} |\RR\eta - m_{\eta,S_0}(\RR\eta)|^2\,d\eta 
\lesssim \ve'\,\Theta_\eta(1.1B_0)^2\eta(S_0),
 \end{equation}
 with $\ve'$ as in Lemma \ref{lem200}.

Now we can apply Theorem \ref{teomain} to the measure $\eta$, and since $\theta_\eta^{n,*}(x)=0$ in $\frac32B_0$, we deduce that
\begin{equation}\label{eqfhvbn3p}
c_1 \Theta_{\eta}(1.1B_0)^2\,\eta(1.1B_0) \leq \int_{\frac32B_0} |\RR\eta - m_{\eta,\frac32B_0}(\RR\eta)|^2\,d\eta.
\end{equation}
If $\ve,\ve_0,\theta_0,\kappa_0$ are small enough (and so $\ve'$ small enough), using that $\eta(\tfrac32B_0)\approx\eta(1.1B_0)$, we get a contradiction with \rf{eqetaM9f}.
\end{proof}

\vv

\subsection{Proof of Theorem \ref{teomain2}}
With the Key Lemma \ref{keylemma} in hand, the proof of Theorem \ref{teomain2} follows by the same arguments as in Section 10 from \cite{Girela-Tolsa}: one constructs a uniformly $n$-rectifiable measure $\zeta$ supported in $CB_0$, with $C\approx1$, which coincides with $\mu$ in a non-negligible piece of $R_0$. More precisely, let
\begin{equation}\label{eqff09}
F = R_0\setminus \bigcup_{Q\in \LD} Q,
\end{equation}
where $R_0$ is the set  is defined
in \rf{eqR0910}.
Then one constructs $\zeta$ as in Section 10 from \cite{Girela-Tolsa}, with the cube $Q_0$ there replaced by $R_0$. Arguing as in Lemmas 10.2 and 10.3 from \cite{Girela-Tolsa}, one proves first that $\zeta$ is  
 $n$-Ahlfors regular, and by comparing it with $\mu$, that $\RR_\zeta$ is bounded in $L^2(\zeta)$, with bounds depending on $\theta_0$ and $\ve_0$. By \cite{NToV}, this implies that $\zeta$ is uniformly $n$-rectifiable.
 For more details, see Section 10 from \cite{Girela-Tolsa}. 
\qed


\subsection{A variant of Theorem \ref{teomain2} which does not rely on Theorem \ref{teoDTlocal00}}\label{secfinal}

In this section we show a weaker variant of Theorem \ref{teomain2} which suffices for some applications to harmonic measure. It has the advantage that its proof is based on Main Lemma \ref{lemaclau}, instead of Theorem \ref{teoDTlocal00}, which in turn relies on the lengthy arguments from \cite{DT}. Remark that the arguments for Main Lemma \ref{lemaclau} require neither Theorem \ref{teoDTlocal00} not the results from
\cite{DT}.

\begin{theorem}\label{teomain2*}
Let $\mu$ be a Radon measure in $\R^{n+1}$ and let $B_0$ be a $(\PP_\mu,C_0)$-doubling ball. Suppose that 
the following conditions hold:
\begin{itemize}
\item[(a)] $\RR_{\mu|_{2B_0}}$ is bounded in $L^2(\mu|_{2B_0})$ with
$$\|\RR_\mu\|_{L^2(\mu|_{2B_0})\to L^2(\mu|_{2B_0})}\leq C_1\Theta_\mu(B_0)$$
and
\begin{equation}\label{eqgrothry63}
\sup_{0<r\leq 4\rad(B_0)} \Theta_\mu(x,r) \leq C_1\Theta_\mu(B_0) \quad\mbox{ for all $x\in 2B_0$}.
\end{equation}
\item[(b')] There exists some ball $B_1$ centered in $B_0$ such that, for some $\delta_1\in (0,1/2)$,
$$\rad(B_1)\leq \delta_1\rad(B_0)$$
and for some  constant $\alpha>0$,
$$\Theta_\mu(\lambda B_1)\geq \alpha\,\Theta_\mu(B_0) \quad\mbox{ for $\lambda\geq 1$ such that $\lambda B_1\subset 2B_0$}.
$$

\item[(c)] For some $\ve>0$,
$$\int_{2B_0} |\RR\mu - m_{\mu,2B_0}(\RR\mu)|^2\,d\mu\leq \ve\,\Theta_\mu(B_0)^2\,\mu(B_0).$$ 
\end{itemize}
Suppose that $\delta_1$ and $\ve$ are small enough, depending on $n$, $C_0$, $C_1$, and $\alpha$. Then
there is some uniformly $n$-rectifiable set $\Gamma$ and some $\tau>0$ such that
$$\mu(\Gamma\cap 2B_0) \geq\tau\,\mu(B_0),$$
with $\tau$ and the uniform rectifiability  constant of $\Gamma$ depending on the above constants.
\end{theorem}

The difference between this theorem and Theorem \ref{teomain2} is that in the condition (b) in Theorem \ref{teomain2} we only require $\Theta_\mu(B_1)\geq \alpha\,\Theta_\mu(B_0)$ while in the condition (b') above we ask $\Theta_\mu(B)\geq \alpha\,\Theta_\mu(B_0)$ for all the balls $B$ concentric
with $B_1$ that contain $B_1$ and are contained in $2B_0$. It is not difficult to check that if, for example, we assume that $\beta_{1,\mu}(B_0)\ll \delta_1\Theta_\mu(B_0),$ then this assumption holds.

\begin{proof}[Proof of Theorem \ref{teomain2*}]
The only difference with the proof of Theorem \ref{teomain2} is that to get \rf{eqfhvbn3p} in Key Lemma \ref{keylemma},
instead of Theorem \ref{teomain}
we apply Main Lemma \ref{lemaclau} to the measure $\eta$,  with $B_0$ replaced by $1.1B_0$, with $R_0$ as in \rf{eqR0910}.

In this situation, the assumptions of Main Lemma \ref{lemaclau} ask that $\eta = h\,\LL^{n+1}$ in $2B_0$ but one can check that in fact it suffices to 
assume this in $S_0$. The condition \rf{eqvar110} in the Main Lemma is an immediate consequence of Lemma \ref{lemDMimproved} and the
growth assumption \rf{eqgrothry63}.
On the other hand, from the condition (b')
in Theorem \ref{teomain2*}, we easily deduce that for any ball $B$ concentric
with $B_1$ containing $2B_1$ and contained in $\frac32B_0$, it holds
$ \Theta_\eta(B)\gtrsim \alpha^{-1}\Theta_\mu(B_0).$
We leave the details for the reader. Then, using again \rf{eqgrothry63} we infer that
$$\int_{B} M_n(\chi_B \eta)\,d\eta \leq C \,\Theta_\mu(B_0)\eta(B)\leq C\alpha^{-1} \Theta_\eta(B)\eta(B),$$
so that $B$ is $(\eta,C\alpha^{-1})$-good. So all the assumptions in Main Lemma \ref{lemaclau} hold, and if $\delta_1$ is small enough (taking $\delta_0=\delta_1$) we get 
$$\int_{S_0} |\RR\eta - m_{\eta,S_0}(\RR\eta)|^2\,d\eta \geq c_1 \Theta_\eta(1.1B_0)^2\,\eta(1.1B_0),$$
which contradicts \rf{eqetaM9f}. 
\end{proof}

\vv

\end{document}